\documentclass{elsarticle}
\usepackage{fullpage}
\usepackage{amsmath}
\usepackage{amsthm}
\usepackage{amsfonts}
\usepackage{amssymb}
\usepackage{amsbsy}
\usepackage{float}
\usepackage{amscd}
\usepackage{sidecap}
\usepackage{cancel}
\usepackage{color}
\newtheorem{theorem}{Theorem}[section]
\newtheorem*{result}{The main result}
\newtheorem{lemma}[theorem]{Lemma}
\newtheorem{prop}[theorem]{Proposition}
\newtheorem{cor}[theorem]{Corollary}
\newtheorem{mydef}[theorem]{Definition}
\newtheorem{remark}[theorem]{Remark}
\newtheorem{hyp}[theorem]{Hypothesis}
\theoremstyle{remark}
\numberwithin{equation}{section}
\def\R{{\mathbb R}}
\def\C{{\mathbb C}}

\begin{document}

\begin{frontmatter}
\title{Finding Eigenvalues the Rupert Way}
\author[unc]{Colin J. Grudzien\corref{cor1}} 
\ead{cgrudz@email.unc.edu}
\address[unc]{Department of Mathematics, University of North Carolina at Chapel Hill, Phillips Hall, CB3250 UNC-CH, Chapel Hill, NC 27599-3250, USA}
\author[Surrey]{Thomas J. Bridges}
\ead{t.bridges@surrey.ac.uk}
\address[Surrey]{Department of Mathematics University of Surrey, Guildford, GU2 7XH, England, UK}
\author[unc]{Christopher K.R.T. Jones}
\ead{ckrtj@email.unc.edu}

\cortext[cor1]{Principal corresponding author}
\begin{abstract}
We develop a stability index for the travelling waves of non-linear reaction diffusion equations using the geometric phase induced on the Hopf bundle $S^{2n-1} \subset \mathbb{C}^n$.  This can be viewed as an alternative formulation of the winding number calculation of the Evans function, whose zeroes correspond to the eigenvalues of the linearization of reaction diffusion operators about the wave.  The stability of a travelling wave can be determined by the existence of eigenvalues of positive real part for the linear operator.  Our \textbf{method of geometric phase} for locating and counting eigenvalues is inspired by the numerical results in Way's \textit{Dynamics in the Hopf bundle, the geometric phase and implications for dynamical systems}\cite{WAY2009}.  We provide a detailed proof of the relationship between the phase and eigenvalues for dynamical systems defined on $\mathbb{C}^2$ and sketch the proof of the method of geometric phase for $\mathbb{C}^n$ and its generalization to boundary-value problems.  Implementing the numerical method, modified from \cite{WAY2009}, we conclude with open questions inspired from the results.
\end{abstract}
\begin{keyword} 
stability analysis; travelling waves; steady states; geometric dynamics; Evans function
\end{keyword}
\end{frontmatter}

\section{Introduction}
\label{section:intro}
Way, in his PhD thesis \cite{WAY2009}, developed numerical results supporting the hypothesis that parallel translation in the Hopf bundle could locate and measure the multiplicity of eigenvalues for linearizations of reaction-diffusion equations, on the real line, about travelling waves.  A generic Hopf bundle is represented as $S^{2n-1}\subset \mathbb{C}^n$, over the base space $\mathbb{C}P^{n-1}$, with fiber $S^1$.  Therefore, any non-zero vector in the space $\mathbb{C}^n$ can be mapped to the Hopf bundle $S^{2n-1}$ via spherical projection.  
This realization of the Hopf bundle as a subset of $\mathbb{C}^n$ allows one to consider an arbitrary complex dynamical system, such as that arising in the Evans function theory,
and map non-zero solutions onto the Hopf bundle.
By constructing our problem appropriately, we may develop a winding number through the parallel translation in the fiber of the Hopf bundle, $S^1$, induced by the dynamics in the phase space.

In particular, the eigenvalue problem for a reaction diffusion operator, linearized about a steady state travelling wave, gives rise to a dynamical system on $\mathbb{C}^n$.  For such linearizations, Way studied the winding in the fiber $S^1$ and its relationship to the eigenvalues of the operator.  Projecting these $\lambda$ dependent, particular solutions onto $S^{2n-1}$ the dynamics on $\mathbb{C}^n$ induce parallel translation in the Hopf bundle.  
The winding in the fiber is called the geometric phase, because of its relationship with Berry's phase in quantum mechanics (e.g.\ Berry~\cite{berry-paper}, Way~\cite{WAY2009}, Chruscinski \& Jamiolkowski \cite{2012geometric}).
In this work we show that particular solutions will pick up information from the dynamics on $\mathbb{C}^n$, and that the winding of these loops of particular solutions can be used to describe the spectrum of the linear operator.  The method of geometric phase is to be considered as a new development of the Evans function that reformulates the eigenvalue calculation.

Consider a system of non-linear reaction diffusion equations, 
	\begin{equation} \begin{matrix} \label{eq:reactiondiffusion}
	U_t=U_{xx}+f(U)\,, & &U(x,0)=U_0(x) \in \mathbb{R}^m\,,
	\end{matrix} \end{equation}
where $f:\R^m\to\R^m$ is a smooth (at least $C^2$) non-linear mapping, and $x\in\R$.
We assume that there exists a travelling wave solution, i.e., a solution of the
single variable $ \xi = x -ct$, so $U(\xi)$ satisfies:
\begin{equation*}\begin{matrix}
-cU'=U''+f(U) &  &\left(' = \frac{d}{d \xi}\right)
\end{matrix}\,.\end{equation*}
Travelling waves and other steady states provide important qualitative understanding of the reaction diffusion equation by describing long time dynamics of the solutions to the PDE.  Stable solutions in particular represent the most physically realistic solutions, being robust with respect to perturbations in the evolution.  The stability of travelling wave solutions for a system as above is determined by the existence of eigenvalues of positive real part for the linearized operator about the wave, as shown in Bates and Jones \cite{Bates89}.
	
The system in equation (\ref{eq:reactiondiffusion}) is re-written in a moving frame as
\begin{align}
\label{eq:wave}
U_t =& U_{\xi \xi} +cU_\xi +f(U)
\end{align}
	for which the travelling wave is a time independent solution.  Linearizing equation (\ref{eq:wave}) about the wave $U(\xi)$, we obtain the $\xi$ dependent operator $\mathcal{L}$ such that:
\begin{equation}\label{eq:linearization}
\mathcal{L} (p)	 = p_{\xi \xi} + c p_\xi +F\big(U(\xi)\big)p \\
\end{equation}
with $p\in\mathbb{B}(\mathbb{R},\mathbb{R}^m)$, the bounded, uniformly continuous functions from $\mathbb{R}$ to $\mathbb{R}^m$, and $F$ the Jacobian of $f$.  

Let $\Omega\subset\mathbb{C}$ be an open, simply connected domain that contains only discrete spectrum of $\mathcal{L}$.  For $\lambda \in \Omega$, we consider the equation
\begin{equation*}
( \mathcal{L} - \lambda I)(p) = 0
\end{equation*}
that has the equivalent formulation as the system
\begin{align*} 
p' =&  q\\
q' =& -cq+\big( \lambda - F(U)\big)p
\end{align*}		
Let $I$ be the $m\times m$ identity matrix---we can write the above as the linear system,  
\begin{equation}\label{eq:matrixsystem}
\begin{matrix}
Y'=A(\lambda,\xi)Y & & Y = \begin{pmatrix} p \\ q \end{pmatrix}\in\mathbb{C}^{2m}
 \\ \\ 
 A(\lambda,\xi) =
\begin{pmatrix}
0 && I\\ \lambda - F(U) && -cI \\
\end{pmatrix}& &
\end{matrix}
\end{equation}
where $A$ is an $n\times n$ complex block matrix where $n\equiv 2 m$.

For dynamical systems of the form (\ref{eq:matrixsystem}) on $\mathbb{C}^n$, the Hopf bundle $S^{2n-1}$ can be realized in the phase space of the system by spherical projection,
\[
\hat{Y} = \frac{Y}{\|Y\|}\in\C^n\cap S^{2n-1}\,.
\]
Reformulating (\ref{eq:matrixsystem}) as a dynamical system on $S^{2n-1}$ generates
paths on the sphere which can then be projected onto ${\mathbb CP}^{n-1}$, with
a phase in $S^1$.

As a property of linear systems any non-zero solution will remain non-zero over finite integration scales, and in this way, the dynamics act naturally on the Hopf bundle.  Loops of solutions in the phase space parametrized in the value $\lambda$ will define parallel translation which, in the fiber $S^1$, yields the winding number.  Provided that the space of solutions satisfying the asymptotic conditions is of dimension one, we use the Hopf bundle in $\mathbb{C}^n$ directly to measure the winding.  Calculating the winding of these particular solutions parametrized in $\lambda$, and measuring this winding relative to asymptotic conditions, we seek to recover the total multiplicity of the eigenvalues enclosed by the $\lambda$ path.

We denote the general approach of calculating the dynamically accumulated winding in the Hopf bundle relative to some asymptotic value as the \textbf{method of geometric phase}---in this work we consider the winding induced on a particular choice of solutions for a reaction diffusion equation.  Our strategy for proving that the geometric phase can be used to count eigenvalues is to relate
it to the Chern number of the bundle formulation of the Evans function in the framework of
Alexander, Gardner \& Jones~\cite{AGJ1990}.  Our work differs from Way's numerical method of geometric phase by realizing the necessity of computing the ``relative phase'' with respect to the asymptotic conditions for the dynamical system---the total accumulated phase of a loop of these particular solutions, relative to the asymptotic conditions, will yield the eigenvalue count. The full development of the method for unbounded systems defined on $\C^2$ is in \S \ref{section:unbounded} and \S \ref{section:twodim}.  For higher dimensional systems we make additional modifications using the exterior algebra and the determinant bundle construction as in Alexander, Gardner \& Jones \cite{AGJ1990}.  Passing to the exterior algebra much of the proof of the method of geometric phase for $\C^2$ holds, and we sketch the proof for general systems on unbounded domains in \S \ref{section:extension}. We also formulate an adaptation of the method of geometric phase to calculate the winding of the Evans function for boundary value problems, and this is treated in \S \ref{section:boundary}.   Our main results are stated in the Theorems \ref{theorem:casei}, \ref{theorem:caseii}, \ref{theorem:caseunbounded} and \ref{theorem:boundary}.  Finally, in \S \ref{section:numerics} we present a numerical example and discuss the implications for future research in \S \ref{section:discussion}.

\section{The Evans function for systems on unbounded domains}
\label{section:unbounded}

The Evans function is a complex-analytic function whose zeros, including multiplicity, correspond to the eigenvalues of $\mathcal{L}$.  The
operator $\mathcal{L}$ is reformulated as a linear, non-autonomous
dynamical system in $\xi$.  This dynamical system has different formulations
depending on whether the interval in $\xi$ is finite or infinite.
We will begin by considering systems on unbounded domains as in Alexander, Gardner \& Jones \cite{AGJ1990}, which serves as the inspiration for the proof of the method of geometric phase.  The theory for the Evans function 
on finite domains was introduced by Gardner \& Jones \cite{GJ91}, and further
developed by Austin \& Bridges \cite{BRI03}.  Hence there will be a
natural extension of the method of geometric phase to boundary value problems and this theory is developed in \S\ref{section:boundary}, starting with the general
boundary value formulation as in \cite{BRI03}.

In both the finite-interval case and infinite-interval case, the problem
reduces to the study of a complex analytic function, the Evans function,
denoted by $D(\lambda)$, whose
zeros correspond to eigenvalues of $\mathcal{L}$.
The strategy for determining stability of the wave with the Evans function is essentially to enclose the eigenvalues for $\mathcal{L}$ of positive real part with some contour in the complex plane, and to use the argument principle with the Evans function to count the zeros enclosed by the contour.  

\subsection{The Evans function}
\label{section:efunct}

  The matrix system (\ref{eq:matrixsystem}) for the eigenvalue problem is non-autonomous with dependence on $U(\xi)$, but the travelling wave solution $U(\xi)$ must be bounded as $\xi \rightarrow \pm \infty$.  Hence, we consider systems such that the travelling wave (\ref{eq:wave}) satisfies the following hypothesis.
\begin{hyp}\label{hyp:wave}
Define the limits of the wave, $\lim_{\xi\rightarrow \pm \infty} U(\xi) = U(\pm \infty)$. We assume that there are positive $a,C\in \mathbb{R}$ for which
\begin{align}
\parallel U(\xi) - U(+\infty) \parallel \leq C e^{-a\xi}  &\hspace{4mm}\text{for}\hspace{4mm}\xi \geq 0 \\
\parallel U(\xi) - U(-\infty)\parallel \leq C e^{a \xi} &\hspace{4mm}\text{for}\hspace{4mm} \xi \leq 0\\
\parallel U'(\xi) \parallel \leq Ce^{-a \mid \xi\mid} &\hspace{4mm}\text{for all }\xi.
\end{align}
\end{hyp}
Under this hypothesis, we may define asymptotic, autonomous systems by the limiting values of the wave:  
	\begin{equation}
	\label{eq:asymptotic}
	\begin{matrix}
	Y'=A_{\pm \infty}(\lambda) Y \\ \\ A_{\pm \infty}(\lambda)  :=  \lim_{\xi\rightarrow \pm \infty} A(\lambda,\xi) = 
	\begin{pmatrix}
	0 & I\\
	\lambda - F\big(U(\pm \infty)\big) & -cI\\
	\end{pmatrix}
	\end{matrix}
	\end{equation}

\begin{mydef}
Let $\mathcal{L}$ be a linear operator derived as in equation (\ref{eq:linearization}) from a non-linear reaction diffusion equation. Suppose the equation $(\mathcal{L}-\lambda)p=0$ defines a flow on $\mathbb{C}^n$ for $\lambda \in \Omega \subset\mathbb{C}$:
\begin{equation}\label{eq:gennonaut} \begin{matrix}
Y' &=& A(\lambda, \xi) Y \\ \\ A_{\pm \infty}(\lambda) &:=& \lim_{\xi \rightarrow \pm \infty} A(\lambda, \xi)\\
\end{matrix} \end{equation}
System (\ref{eq:gennonaut}) is said to \textbf{split} in $\Omega$ if $A_{\pm \infty}$ are hyperbolic and each have exactly $k$ eigenvalues of positive real part (\textbf{unstable eigenvalues}) and $n-k$ eigenvalues of negative real part (\textbf{stable eigenvalues}), including multiplicity, for every $\lambda \in \Omega$.  
\end{mydef}

\begin{hyp}\label{hyp:split}
Assume $\Omega$ is open, simply connected and contains only discrete eigenvalues of $\mathcal{L}$.  We always assume that the systems under consideration split in the domain $\Omega$.
\end{hyp}

\begin{hyp}\label{hyp:K} Let $K\subset \mathbb{C}$ be a contour in $\mathbb{C}$, describing a path for the spectral parameter $\lambda$.  Assume that the contour $K$ is a piecewise smooth, simple closed curve in $\Omega \subset \mathbb{C}$ such that there is no spectrum of $\mathcal{L}$ in $K$.  Let $K^\circ$ be the region enclosed by $K$---we assume $K^\circ$ is homeomorphic to the disk $D\subset\mathbb{R}^2$ and that $K$ is parametrized by $\lambda(s):[0,1]\hookrightarrow K$ with standard orientation.
\end{hyp}

The above hypotheses will allow us to construct the Evans function on unbounded domains; the Evans function was first derived in a series of papers \cite{evans1972},\cite{evans72},\cite{evan72},\cite{evans75} by Evans on nerve impulse equations, and was generalized by Alexander, Gardner \& Jones \cite{AGJ1990} for general systems of reaction diffusion equations.  The Evans function has been applied in many more situations and its development as well as the current state-of-the-art is well documented and explained in Kapitula \& Promislow \cite{kapitula2013}.
 
Our work in this paper is to reformulate the winding number calculation for the Evans function into a new geometric setting, as was suggested by Way \cite{WAY2009}.  The construction of the Evans function by Alexander, Gardner \& Jones \cite{AGJ1990} utilizes a vector bundle construction to take advantage of the unique classification of complex vector bundles over 2-spheres with their Chern numbers.  We frame our discussion in this setting for the Evans function to establish the link between Chern numbers and the method of geometric phase.

\subsection{The unstable bundle}
Recall that the eigenfunctions for $\mathcal{L}$, as in equation (\ref{eq:linearization}), are required to be bounded for all $\xi\in\R$. 
For the associated system of equations (\ref{eq:asymptotic}), the eigenvalues of $A_{\pm \infty}$ determine the asymptotic growth and decay rates of potential
eigenfunctions.  By a compactification of the $\xi$ parameter we may define a dynamical system for $\{\xi \in [-\infty,+\infty]\}$ ``capped'' on the ends by these asymptotic, autonomous systems.  The asymptotic systems have fixed points at $0$, by linearity of the dynamics, and thus we define un/stable manifolds of the extended system.  The un/stable eigenvectors of the system at $\pm \infty$ determine the asymptotic behavior of solutions that lie in the un/stable manifolds of the critical points of the asymptotic systems. 
\begin{lemma}
Under the above hypotheses, \ref{hyp:wave} and \ref{hyp:split}, a solution to the extended system is an eigenfunction for $\mathcal{L}$ if and only if it is in the unstable manifold for $A_{-\infty}$ and the stable manifold of $A_{+\infty}$.
\end{lemma}
\begin{proof}
This is proved by Alexander, Gardner \& Jones \cite{AGJ1990}.\end{proof}
We define the $\xi$ dependent variable $\tau$ to compact the dynamics, where
\begin{align*}
\xi  &= : \frac{1}{2\kappa} \log\left(\frac{1+\tau}{1-\tau}\right)
\end{align*}
for some $\kappa\in \mathbb{R}$.  Appending the $\tau$ yields the new system
\begin{equation}\label{eq:taudependent} \begin{matrix}
Y' = A(\lambda,\tau) Y   & &
A(\lambda, \tau) = \begin{cases}
A(\lambda, \xi(\tau)) &  \text{for $\tau\neq \pm 1$} \\
A_{\pm \infty}(\lambda) &  \text{for $\tau = \pm 1$}
\end{cases} \\ \\
 \tau' =\kappa(1-\tau^2)&  & '=\frac{d}{d\xi}   
\end{matrix} \end{equation}

\begin{lemma}
We may choose $\kappa>0$ such that the flow defined by equation (\ref{eq:taudependent}) is $C^1$ on the entire compact interval.  
\end{lemma}
\begin{proof}
On finite time scales the flow $(\ref{eq:taudependent})$ is smooth by linearity, but Lemma 3.1 in Alexander, Gardner \& Jones \cite{AGJ1990} shows that if $\kappa<\frac{a}{2}$, where $a$ is defined in Hypothesis \ref{hyp:wave}, then equation (\ref{eq:taudependent})
$C^1$ on the entire compact interval.\end{proof}

\begin{hyp}
We will assume that for all systems $0 < \kappa < \frac{a}{2}$. 
\end{hyp}

Within the invariant planes $\{ \tau = \pm 1\}$ of system (\ref{eq:taudependent}), the dynamics are governed by the linear, autonomous equations
\begin{equation*} \begin{matrix}
Y'=A_{\pm\infty}Y & & ' = \frac{d}{d\xi} \\ \\
\tau' \equiv 0 & & (\tau = \pm 1) \\
\end{matrix} \end{equation*}
so that solutions in these planes are determined entirely by the stable and unstable directions of the asymptotic systems.  For $\{\tau \in (-1,+1) \}$, solutions are governed by the non-autonomous system and have limits in the invariant planes as $\xi \rightarrow \pm\infty$. 

Consider the un/stable manifolds of the critical points
\begin{equation*}\begin{matrix}
(0,\pm1) \in \mathbb{C}^n \times \{\tau = \pm 1\}
\end{matrix} \end{equation*} 
The dynamics in the invariant planes are linear with $k$ unstable directions and $n-k$ stable directions; with the appended $\tau$ equation, the system gains one real unstable/ stable direction at $\tau = \mp 1$ respectively.  Standard invariant manifold theory dictates that there is a $2k+1$ \textbf{(real)} dimensional local unstable manifold in some neighborhood of $(0,-1)$ that can be extended globally by taking its flow forward for all time.  In the invariant plane $\tau=-1$, the unstable manifold is just the span of the unstable eigenvectors, but for $\tau>-1$, this becomes a $\tau$ dependent subspace of $\mathbb{C}^n$.

From the contour $K$ and the $\tau$ parameter we construct a ``parameter sphere'' above which we can view solutions to the system in equation (\ref{eq:taudependent}) as paths in an appended trivial $\mathbb{C}^n$ bundle.  
\begin{mydef}
The set $K\times \{\tau \in [-1,+1]\}$
 defines a topological cylinder as $K$ is topologically equivalent to $S^1$.  Gluing $K^\circ$, the region enclosed by $K$, to the cylinder with we obtain a topological 2-sphere, 
 \begin{equation}\label{eq:parametersphere}\begin{matrix}
 M\equiv K\times \{\tau \in [-1,+1] \} \cup K^\circ \times \{\tau = \pm 1\}
 \end{matrix}\end{equation} 
 hereafter denoted the \textbf{parameter sphere}.  The trivial $\mathbb{C}^n$ bundle over the parameter sphere is defined as $M \times \mathbb{C}^n$.
\end{mydef}  
  Solutions to the system in equation (\ref{eq:taudependent}) can be tracked in the fibers of the trivial bundle, with their evolution defined by the flow and the parameter values in $M$.  Alexander, Gardner \& Jones \cite{AGJ1990} show that for fixed $\lambda \in K$, the unstable manifold of the critical point $(0,-1)\in\mathbb{C}^n\times[-1,1]$ converges to the unstable space of $A_{+\infty}(\lambda)$ for $\tau =1$ in Grassmann norm.  We will foliate the unstable manifold over $\{ \lambda \in K^\circ \} \times \{\tau = +1\}$ with fibers defined by the unstable subspace of $A_{+\infty}(\lambda)$.  
\begin{mydef}
The unstable manifold of the critical point $(0,-1) \in\mathbb{C}^n\times[-1,1]$ defines a subspace of $\mathbb{C}^n$ that approaches $(0,-1)$, exponentially decaying as $\xi \rightarrow - \infty$ for $\mid\xi\mid$ sufficiently large.  For each fixed $(\lambda,\tau)$ let $W^u (\lambda, \tau)$ denote the \textbf{unstable manifold} in $\mathbb{C}^n$ defined by the flow at $(\lambda,\tau)$. The 
total space $E$ defines a non-trivial bundle over $M$ with projection
$\pi_E:E\to M$,
\begin{equation}\begin{CD}
W^u @>>> E\\
& & @VVV \pi_E \\
& & M
\end{CD}\end{equation}  
E is contained in the trivial bundle $M\times \mathbb{C}^n$,
and is called the \textbf{unstable bundle}.
\end{mydef}
\begin{lemma}
  The unstable bundle is a $k$ dimensional vector bundle over the sphere $M$.
\end{lemma}
\begin{proof}
For a proof the reader is referred to Alexander, Gardner \& Jones \cite{AGJ1990}.
\end{proof}

The construction of the unstable bundle is useful because the sum of its Chern numbers is related to the total multiplicity of the eigenvalues enclosed by the contour $K$.  Chern numbers represent toplogical invariants for a complex vector bundle, and there are several ways to treat their derivation---we provide only a cursory description of the Chern numbers of the unstable bundle. 
\begin{lemma}
Given a connection $\omega$ for the unstable bundle we can construct the curvature form for the bundle by the relation
\begin{align}
d\omega = -\omega \wedge \omega + \Omega
\end{align} 
where $\Omega$ is the curvature form.
Let $H^j(M,\mathbb{Z})$ be the $j^{th}$ cohomology group of the parameter sphere with coefficients in $\mathbb{Z}$.  The Chern class of degree $j$ for the parameter sphere is an element
\begin{align}
C_j(E) \in H^{2j}(M,\mathbb{Z}),
\end{align}
and is the $j^{th}$ coefficient of the characteristic polynomial of the curvature form $\Omega$, ie:
\begin{align}
\det\left(I + \frac{t}{2 \pi i} \Omega\right) = 1 + \sum^k_{j=1} t^{j} C_j(E)
\end{align}
\end{lemma}
\begin{proof}
This is simply the application of classical results to this specific bundle construction---for a discussion of the general results and a derivation of character classes for general vector bundles consult Morita \cite{morita2001}.
\end{proof}
\begin{remark}
The Chern classes are independent of the choice of the connection $\omega$.
The \textbf{Chern number}, denoted $c_1(E)$,
is the integral over the base manifold of an element in the Chern class.
\end{remark}
\begin{cor}
For the $k$ dimensional unstable bundle, $c_1(E)$ is the only non-trivial Chern number.
\end{cor}
\begin{proof}
Recall that the parameter sphere $M\cong S^2$ and the cohomology groups are given
\begin{align}
H^j(S^2,\mathbb{Z}) \cong \begin{cases} \mathbb{Z}\text{ if $j=0,2$} \\ 0 \text{ otherwise} \end{cases}
\end{align}
\end{proof}
\begin{remark}
As $c_1(E)$ is the only non-trivial Chern number of the $k$ dimensional unstable bundle, we may describe $c_1(E)$ unambiguously as \textbf{the Chern number} of the unstable bundle.
\end{remark}
\begin{lemma}
The Chern number of the unstable bundle equals the total multiplicity of the eigenvalues enclosed by the contour $K$.
\end{lemma}
\begin{proof}
This is one of the essential results of Alexander, Gardner \& Jones \cite{AGJ1990} and the reader is referred there for details, and a full development of the unstable bundle.
\end{proof}

\subsection{The Hopf bundle and the method of geometric phase}

The Hopf bundle is a classical example of a principal fiber bundle, which has a realization in $\mathbb{C}^n$---we take advantage of this realization to re-frame the winding of the unstable bundle in terms of the geometric phase induced in the fibers.
\begin{mydef}
The Hopf bundle is a principal fiber bundle with full space $S^{2n-1}$, base space $\mathbb{C}P^{n-1}$, and fiber $S^1$. The fiber $S^1$ acts naturally on $S^{2n-1}$ by the action of the unitary group $U(1)$; with respect to this action the quotient is $\mathbb{C}P^{n-1}$.  The spaces are related by the diagram
\begin{equation}\begin{CD}
S^1 @>>> S^{2n-1}\\
& & @VVV \pi \\
& & \mathbb{C}P^{n-1}
\end{CD}\end{equation}  
where $\pi$ is the quotient map induced by the group action.
\end{mydef}

 For a generic Hopf bundle, of dimension $2n-1$, there exists an intuitive choice of connection between fibers.  We will use the realization of $S^{2n-1}\subset \mathbb{C}^n$ by spherical projection to define the connection pointwise.

\begin{mydef}
For the Hopf bundle $S^{2n-1}$, viewed in coordinates for $\mathbb{C}^n$, we define the connection 1-form $\omega$ pointwise for $p\in S^{2n-1}$ as a mapping of the tangent space of the Hopf bundle $T_p\left(S^{2n-1}\right)\subset T_p\left(\mathbb{C}^n\right)$
\begin{equation} \begin{matrix}
\omega_p : & T_p\left(S^{2n-1}\right) & \rightarrow & i\mathbb{R}  \\ \\
 & V_p & \mapsto & \langle V_p, p \rangle_{\mathbb{C}^n} \\
\end{matrix} \end{equation}
where $i\mathbb{R}$ is the Lie algebra of the fiber $S^1$ \cite{WAY2009}.  A connection 1-form defines a connection and we denote $\omega$ to be the \textbf{natural connection} on the Hopf bundle.
\end{mydef}
\begin{lemma}
The natural connection is a connection of the generic Hopf bundle $S^{2n-1}$ and it is the unique connection for the $S^3$ Hopf bundle.
\end{lemma}
\begin{proof}
This is proven by Way \cite{WAY2009} in \S 3.5 and the reader is referred there for a full discussion.\end{proof}
A choice of connection decomposes the tangent space of the Hopf bundle into horizontal and vertical subspaces, which introduces the concept of parallel translation in the fibers.  
\begin{mydef}
The \textbf{vertical} subspace is always canonically defined by the kernel of the push forward of the projection map $\pi$ onto the base space. The \textbf{horizontal} subspace is transverse to the vertical subspace and 
isomorphic to the tangent space of the base space, but in general is not unique, and is defined by the choice of the connection.  In particular, the horizontal subspace can be defined as the kernel of the connection 1-form.  

The vertical and horizontal subspaces of the tangent space, by choice of a connection, give a smooth decomposition of the full tangent space
\begin{align*}
T\left(S^{2n-1}\right) = V\left(S^{2n-1}\right) \oplus H\left(S^{2n-1}\right)
\end{align*}
over the base space $S^{2n-1}$, which is compatible with the trivializations of the bundle.  
\end{mydef}

\begin{remark}
The connection described for the Hopf bundle is a connection on a principal fiber bundle---similarly connections on vector bundles can be described with respect to a group action and the appropriate spaces, but this goes beyond the scope of our discussion.  For a full discussion of vertical and horizontal subspaces, and the theory of connections, the reader is referred to Kobayashi \& Nomizu \cite{kobayashi1996}. 
\end{remark}
Given a differentiable path in the Hopf bundle, and a choice of connection, we may always choose a corresponding ``horizontal lift'', which will describe the displacement in the fiber.  We define the horizontal lift in a similar vein as Kobayashi \& Nomizu \cite{kobayashi1996}. 
\begin{mydef}
Let $v(s) : [0,1] \rightarrow S^{2n-1}$
be a differentiable path in the Hopf bundle.  The \textbf{horizontal lift} of $v(s)$ is a path $w(s): [0,1] \rightarrow S^{2n-1}$ for which 
\begin{align*}
w(0) = v(0)& \\
\pi\big(w(s)\big) \equiv \pi\big(v(s)\big)&\hspace{3mm} \forall s \\
\omega\left(\frac{d}{ds}(w(s)\right) \equiv 0&\hspace{3mm} \forall s.
\end{align*}
ie: $\frac{d}{ds} w(s) \in H\left(S^{2n-1}\right)$ for all $s$.
\end{mydef}
  A differentiable path in the Hopf bundle with a tangent vector that is always in the horizontal subspace experiences no motion in the fiber---parallel translation consists of the displacement in the fibers between a path and its horizontal lift. The fiber of the Hopf bundle is the circle, so parallel translation induces a natural winding number through the displacement in the fiber. 
\begin{mydef}\label{def:geophase}
Let $v(s)$ be a differentiable path in the Hopf bundle, $v: [0,1] \mapsto S^{2n-1}$, and let $w(s)$ be its horizontal lift.
The \textbf{phase curve} $\theta(s)$ for $v(s)$ is defined by the equation
\begin{align}
v(s)&= e^{i\theta(s)}w(s)
\end{align}
ie: the path in the fiber describing the displacement along $v(s)$ between $v(s)$ and its horizontal lift $w(s)$. The \textbf{geometric phase} is the change in the phase curve, ie:
\begin{align}
GP\Big(v\big([0,1]\big)\Big) \equiv & \frac{\theta(1) - \theta(0)}{2\pi}
\end{align}
\end{mydef}
\begin{lemma}
Let $v(s)\subset S^{2n-1}$ parametrize a smooth path $\Gamma$ in the Hopf bundle for $s\in[0,1]$, and let $\theta(s)$ be the phase curve with respect to the horizontal lift $w(s)$.  Then the phase curve satisfies the differential equation
\begin{align}
\theta'(s) = -i \omega\big(v'(s)\big) & & \theta(0) = 0
\end{align}
and the geometric phase can be computed as the pull back of the connection 1-form along $\Gamma$, ie:
\begin{align}
\label{eq:geometricphase}\frac{\theta(1)}{2\pi} & = \frac{1}{2\pi i}\int_\Gamma \omega \\
                       & = \frac{1}{2\pi i}\int^1_{0} \big\langle v'(s), v(s)\big\rangle ds
\end{align}  
\end{lemma}
\begin{proof}
The general form of the differential equation describing the phase curve is derived by Kobayashi \& Nomizu \cite{kobayashi1996}, and is formulated with respect to the natural connection on the Hopf bundle by Way \cite{WAY2009}.  
\end{proof}
\begin{remark}
The geometric phase has important connections to the Berry phase in quantum mechanics, discussed by Way \cite{WAY2009}, and Chruscinski \& Jamiolkowski \cite{2012geometric}.
\end{remark}

The method for computing eigenvalues with geometric phase utilizes general \textbf{non-zero}, differentiable paths in $\mathbb{C}^n$---a useful reformulation of the phase integral in equation (\ref{eq:geometricphase}) for non-zero paths is given in the following lemma.
\begin{lemma}\label{lemma:phaseintegrals}
Suppose for $s\in[0,1]$, $u(s)$ is a non-zero, differentiable path in $\mathbb{C}^n$.  Then the connection of its spherical projection, $\hat{u}(s)\in S^{2n-1}$, can be written
\begin{align}
\omega\left( \frac{d}{ds} \hat{u}(s)\right) =& i\frac{Im\Big(\big\langle u'(s), u(s) \big\rangle\Big)}{\big\langle u(s),u(s) \big\rangle}\label{eq:altconnection}
\end{align}
and the geometric phase along $\hat{u}(s)$ can be computed as
\begin{align}
\frac{\theta(1)}{2\pi} \label{eq:computephase}
&=\frac{1}{2\pi } \int^1_0  \frac{Im\Big(\big\langle u'(s), u(s) \big\rangle\Big)}{\big\langle u(s),u(s) \big\rangle} ds
\end{align}
If $u(s)$ is also a \textbf{closed curve}, then 
\begin{align}
0 \label{eq:realphase}
&= \int^1_0  \frac{Re\Big(\big\langle u'(s), u(s) \big\rangle\Big)}{\big\langle u(s),u(s) \big\rangle} ds
\end{align}
and the geometric phase is equivalent to 
\begin{align}\label{eq:completeintegral}
\frac{\theta(1)}{2\pi} &= \frac{1}{2\pi i} \int^1_0 \frac{\big\langle u'(s),u(s)\big\rangle}{\big\langle u(s), u(s)\big\rangle} ds
\end{align}

\end{lemma}
\begin{proof}
We will begin by deriving the alternative form of the connection (\ref{eq:altconnection}).  If $\hat{u}(s)$ is the spherical projection of the path $u(s)$, then the natural connection is identically
\begin{align*}
\omega\left(\frac{d}{ds} \hat{u}(s)\right) &= \left\langle \frac{d}{ds} \frac{u(s)}{\big\langle u(s),u(s)\big\rangle^\frac{1}{2}}, \frac{u(s)}{\big\langle u(s),u(s)\big\rangle^\frac{1}{2}} \right\rangle\\
&= \left\langle  \frac{u'(s)\big\langle u(s),u(s)\big\rangle^\frac{1}{2}}{\big\langle u(s),u(s)\big\rangle}  - \frac{ u(s) Re\Big(\big\langle u'(s),u(s)\big\rangle\Big)}{\big\langle u(s),u(s)\big\rangle^\frac{3}{2}}, \frac{u(s)}{\big\langle u(s),u(s)\big\rangle^\frac{1}{2}} \right\rangle\\
&=   \frac{\big\langle u'(s),u(s)\big\rangle}{\big\langle u(s), u(s)\big\rangle} - \frac{ Re\Big(\big\langle u'(s),u(s)\big\rangle\Big)}{\big\langle u(s),u(s)\big\rangle}\\
&= i\frac{ Im \Big( \big\langle u'(s), u(s) \big\rangle\Big)}{\big\langle u(s),u(s)\big\rangle}  
\end{align*}
which verifies the equations (\ref{eq:altconnection}) and (\ref{eq:computephase}).  Suppose that $u(s)$ is also a closed curve---then notice,
\begin{align*}
0 & = \log \left(\parallel u(s) \parallel^2 \right) \Big|^{s=1}_{s=0} \\
  & = \int^1_0 \frac{\frac{d}{ds} \big\langle u(s), u(s) \big\rangle }{\big\langle u(s), u(s)\big\rangle} ds\\
  & = 2  \int^1_0 \frac{Re\Big(\big\langle u'(s), u(s)\big\rangle\Big) }{\big\langle u(s), u(s)\big\rangle} ds
\end{align*}
which verifies equation (\ref{eq:realphase})---combining this with equation (\ref{eq:computephase}) this verifies equation (\ref{eq:completeintegral}). 
\end{proof}
\begin{remark}
Given the formulation (\ref{eq:computephase}) of the geometric phase in terms of any non-zero path, we may unambiguously refer to the \textbf{geometric phase of a path} $\mathbf{u(s)\in\mathbb{C}^n}$, describing the geometric phase of its normalization.  
\end{remark}

\subsection{The method of geometric phase on $\mathbb{C}^2$}
We develop the method of geometric phase first in the case where the dynamical system is defined on $\mathbb{C}^2$, where the low dimension allows geometric intuition.  This intuition is useful in proving the general technique, and much of the argument is identical for systems of larger dimension once we introduce determinant bundle.  The reader can consider the scalar bistable equation as a typical example of a PDE for which $\mathcal{L}-\lambda=0$ defines a system on $\mathbb{C}^2$ satisfying the Hypotheses \ref{hyp:wave} and \ref{hyp:split}:
\begin{align}\label{eq:bistable}
u_t=u_{xx}+f(u) &  & f(u)=u(u+1)(u-1) 
\end{align}
This PDE has steady localized solutions, and the spectral problem associated with 
the linearization about such a state can be
formulated as in (\ref{eq:matrixsystem}) with $Y\in\C^2$.
We will revisit this example in \S\ref{section:numerics} and present results demonstrating the numerical method.  The method of geometric phase for such a PDE defining an ODE system on $\mathbb{C}^2$ is described as follows.
\vspace{.25cm}

\begin{minipage}[T]{\linewidth}
\begin{center}
\textbf{Table 1: The Method of Geometric Phase on $\mathbb{C}^2$}\\[3mm]
\begin{tabular}{r|l}
\hline
\phantom{$\bigg|$}\textbf{Step 1:} & 
Choose a contour $K$ in $\mathbb{C}$ that does not intersect the spectrum of the operator $\mathcal{L}$.\\ 
\hline
\phantom{$\bigg|$}\textbf{Step 2:} &
Varying $\lambda \in K$ define $X^+(\lambda)$ to be an analytic loop of eigenvectors for the $A_{+\infty}(\lambda)$ system in\\ & equation (\ref{eq:gennonaut})
where $X^+(\lambda)$ corresponds to the eigenvalue of positive real part.\\ 
\hline
\phantom{$\bigg|$}\textbf{Step 3:} & 
Suppose $Z\big(\lambda,\tau(\xi)\big)$ is a solution to the system defined by equation (\ref{eq:gennonaut}), such that \\ &$\Big(Z\big(\lambda,\tau(\xi)\big),\tau\Big)$ is in the unstable manifold $(0,-1)\in \mathbb{C}^2 \times [-1,1]$ for
equation (\ref{eq:taudependent}).\\
\hline
\phantom{$\bigg|$}\textbf{Step 4:} & 
Calculate the \textbf{relative geometric phase} of $Z\big(\lambda,\tau(\xi)\big)$ with respect to $X^+(\lambda)$,\\& ie:
$GP\Big(Z\big(K,\tau(\xi)\big)\Big) - GP\Big(X^+\big(K\big)\Big)$,
where $GP\Big(u\big([0,1]\big)\Big)$ is the geometric phase of a\\& 
non-zero, differentiable path in $\C^2$, defined in equation (\ref{eq:computephase}). 
\\ 
\hline
\end{tabular}
\end{center}
\end{minipage}
\vspace{\belowdisplayskip}

\begin{result}
The central theme of this work is demonstrating that, for an appropriate choice of $X^+(\lambda)$ and $Z(\lambda,\tau)$, \textbf{the asymptotic relative phase}
\begin{align}\label{eq:asymptoticrelphase}
\lim_{\xi\rightarrow\infty}GP\Big(Z\big(K,\tau(\xi)\big)\Big) - GP\Big(X^+(K)\Big)
\end{align}
equals the total multiplicity of the eigenvalues enclosed by $K$.
\end{result}
Way's numerics supported the hypothesis that the geometric phase of $Z\big(\lambda,\tau(\xi_1)\big)$ should equal the total multiplicity of the eigenvalues for $\mathcal{L}$ in $K^\circ$ when $\xi_1$ is taken sufficiently large \cite{WAY2009}.  However, in our study we reformulate his idea with our \textbf{asymptotic relative phase} calculation, in equation (\ref{eq:asymptoticrelphase}), and the machinery of the determinant bundle.  The dependence on the eigenvectors for $A_{+\infty}(\lambda)$ in the computation of the \textbf{relative phase} turns out to be an essential point in formulating the method, as is using the determinant bundle.  The original numerical method studied the geometric phase of a single eigenvector corresponding to the strongest growing/decaying eigenvalue but in general the information of the full un/stable subspace is required.  We take advantage of the existing computation of the total multiplicity of eigenvalues for $\mathcal{L}$ through the Chern number for the determinant bundle and prove method of geometric phase in generality for unbounded domains in this framework.  Our work was thus to prove the full relationship between the geometric phase and eigenvalues and to adapt the relative phase calculation for higher dimensions and more general boundary conditions.  We return to the example (\ref{eq:bistable}) in \S\ref{section:numerics} to demonstrate the technique as described in the steps above, and explore new questions inspired from the results.

\section{Unstable bundle and the two-dimensional case}
\label{section:twodim}
In this section we will restrict to the case where $m=1$ in the
system (\ref{eq:matrixsystem}) and to the case where the asymptotic system
is symmetric,
\begin{equation*}
\lim_{\xi \rightarrow \pm \infty} A(\lambda, \xi) \equiv A_{\pm\infty}(\lambda) \equiv A_\infty(\lambda)\,.
\end{equation*}  
This restriction on the boundary conditions will give useful geometric intuition of the method, but the restriction is not necessary in general.  We will adapt the theory and proofs presented in \S\ref{section:twodim} to the
general construction of the unstable bundle for $n$ dimensions, $k$ unstable directions, and non-symmetric asymptotic limits in \S \ref{section:extension}.
\subsection{Set up of the Evans function system on $\mathbb{C}^2$}
Let $\mathcal{L}$ be the linearization of a reaction diffusion equation about a steady state.  From $(\mathcal{L}-\lambda)p=0$, where $\lambda \in \Omega \subset\mathbb{C}$, we derive the system on $\mathbb{C}^2$:

\begin{equation}\label{eq:symmetric} \begin{matrix}
Y' &=& A(\lambda, \tau) Y & & A_{\infty}(\lambda) := \lim_{\xi \rightarrow \pm \infty} A(\lambda, \tau)\\ \\
\tau ' &=& \kappa ( 1- \tau^2) & & \\ \\
 A(\lambda, \tau) &=& \begin{cases}
A\big(\lambda, \xi(\tau)\big) &  \text{for $\tau\neq \pm 1$} \\ \\
A_{\infty}(\lambda) &  \text{for $\tau = \pm 1$} \\
\end{cases} \end{matrix}
\end{equation}
where $\Omega$ is open and simply connected, and the system at infinity, $A_\infty(\lambda)$, has one stable and one unstable eigenvalue for every $\lambda \in \Omega$.  Let $K$ be a smooth, simple closed curve in $\Omega \subset \mathbb{C}$ that contains no spectrum of $\mathcal{L}$, let the enclosed region be denoted $K^\circ$ and let $K$ be parametrized by $\lambda(s):[0,1]\hookrightarrow K$.

Denote the eigenvalues of $A_\infty(\lambda)$ by $\mu_1(\lambda)$,$\mu_2(\lambda)$ with
\begin{equation*}Re(\mu_1) < 0 < Re(\mu_2)
\end{equation*} 
for each $\lambda\in \Omega$.  The vector
\begin{equation*} 
X := e^{-\mu_2(\lambda) \xi} Y
\end{equation*}
is in $W^u(\lambda, \tau )$ provided $Y \in W^u(\lambda, \tau)$, because $W^u(\lambda,\tau)$ is a subspace. Let $\frac{d}{d\xi} = '$, then
\begin{align*}
X' =&  -\mu_2(\lambda) e^{- \mu_2(\lambda) \xi} Y + e^{-\mu_2 (\lambda) \xi} Y'\\
 =& (A - \mu_2 I) X
 \end{align*}
 This motivates the following system on $\mathbb{C}^2$:
 \begin{equation}\label{eq:bsymmetric} \begin{matrix}
 X' &=& BX  & & B(\lambda,\tau) := \big(A(\lambda,\tau) - \mu_2 (\lambda)I\big) \\ \\
 \tau ' &=& \kappa ( 1- \tau^2) & & B_\infty (\lambda) := \lim_{\xi \rightarrow \pm \infty} B(\lambda,\tau)
 \end{matrix} \end{equation}
The $\xi$ dependent rescaling transforms the $A$ system in equation (\ref{eq:symmetric}) into the $B$ system in equation (\ref{eq:bsymmetric}) where it will be more convenient to work with the trajectories in the unstable manifold.  
 
 \begin{mydef}Solutions to the $A$ system (\ref{eq:symmetric}) will be denoted with a $\sharp$.  That is, if $Z^\sharp \in W^u(\lambda, \tau_0)$, then there is a $\xi_0$ for which $Z \equiv e^{-\mu_2(\lambda) (\xi - \xi_0)}Z^\sharp$ is the unique solution to the $B$ system (\ref{eq:bsymmetric}) that agrees with $Z^\sharp$ at $(\lambda, \tau_0)$.  Similarly if $Z$ is a solution to the $B$ system, then $Z^\sharp \equiv e^{\mu_2(\lambda) (\xi -\xi_0)} Z$ is the unique solution to the $A$ system that agrees with $Z$ at $\xi_0$.
 \end{mydef}
 
Let $X^-(\lambda)$ be an unstable eigenvector for $A_{-\infty}(\lambda)$.  Then in the $B$ system (\ref{eq:bsymmetric}), if $Z\in span_{\mathbb{C}} \{ X^- (\lambda)\}$, then $(Z,\pm 1)$ is a fixed point.  By the construction of $B_\infty$, in $\tau = - 1$, there is exactly one complex stable direction, one complex center direction corresponding to the line of fixed points, and the real unstable $\tau$ direction.  We may thus construct the center-unstable manifold of a non-zero path of eigenvectors $X^-(\lambda)$ that correspond to the zero eigenvalue in the $B(\lambda)$ system (\ref{eq:bsymmetric}).
 \subsection{The induced flow on $S^3$}
In order to measure the geometric phase of a solution which spans the unstable subspace $W^u(\lambda,\tau)$ we must project the solution onto $S^3$.  On finite timescales, ie: $\tau\in(-1,1)$, this isn't an issue.  A non-zero solution to equation (\ref{eq:symmetric}) may be viewed in hyper-spherical coordinates
\begin{equation*}\{r\in(0,+\infty)\} \times S^3 \times \{ \tau\in(-1,1) \}\end{equation*}
because no solution reaches zero in finite time. However, to measure the phase over the entire bundle we must appeal to solutions to the $B$ system; Lemma 3.7 in Alexander, Gardner \& Jones \cite{AGJ1990} tells us that a solution to equation (\ref{eq:symmetric}) that is in $W^u$, is unbounded and converges to the unstable subspace of $(0,+1)$ in the Grassmann norm as $\xi \rightarrow +\infty$.  Because the solutions of the system (\ref{eq:symmetric}) in $W^u$ approach $0$ as $\xi\rightarrow -\infty$ and are unbounded as $\xi \rightarrow \infty$, we appeal to solutions of the $B$ system instead.

\begin{lemma} There exists a choice of unstable eigenvectors for $A_{\pm\infty}(\lambda)$, $X^\pm(\lambda)$, analytic in $\lambda$ for $\lambda \in \Omega$. 
\end{lemma}
\begin{proof}
For a constructive algorithm for such bases the reader is referred to Humpherys, Sandstede \& Zumbrun \cite{hump06}.\end{proof}
Note that under spherical projection we may lose $\mathbb{C}$ differentiability, but we will retain the differentiability in $s$, where $\lambda(s):[0,1] \hookrightarrow K$ and $s$ is the path parameter.  
\begin{mydef}
Let the contour $K\subset \mathbb{C}$ be given.  A \textbf{reference path} for $\lambda\in K$, defined $X^\pm(\lambda)$ at $\tau =\pm1$ respectively, is a loop of eigenvectors for $A_{\pm\infty}(\lambda)$  that corresponds to the eigenvalue of \textbf{largest, positive, real part} for $A_{\pm\infty}$.
\end{mydef}

\begin{mydef}
Let $X^\pm(\lambda)$ be a reference path chosen analytically in $\lambda$ over $K$ that can be extended smoothly over $K^\circ$ without zeros.  $X^\pm(\lambda)$ is denoted \textbf{non-degenerate} as $X^\pm(\lambda)$ defines fibers compatible with the unstable bundle construction.
\end{mydef}

\begin{lemma}\label{lemma:nonsingproj}
Let $X^-(\lambda)$ be a non-degenerate reference path for $A_{-\infty}$.  Let the center-unstable manifold of this line of critical points, in the $B$ system (\ref{eq:bsymmetric}), be parametrized by $(\lambda,\tau)$ as $Z(\lambda,\tau)$.  Then $Z(\lambda,\tau)$ is non-singular and continuous in its limit $\xi \rightarrow +\infty$, and the span equals the unstable manifold $W^u(\lambda,\tau)$ for all $(\lambda,\tau) \in K\times [-1,1]$.
\end{lemma}  
\begin{proof}
As in \S 4 of Alexander, Gardner \& Jones \cite{AGJ1990}, the center-unstable manifold of the path $X^-(\lambda)$ in the $B$ system can parametrized by $(\lambda,\tau)$
\begin{equation*} \begin{matrix} 
 Z(\lambda,\tau) & & Z(\lambda,-1) \equiv X^-(\lambda) \\ \\
 \end{matrix} \end{equation*}
 such that it is $\mathbb{C}$ differentiable in $\lambda$ for $\tau \in [-1,1)$ fixed.  
  
  The $\xi$ dependent scaling of $Z$  
 \begin{equation*}Z^\sharp(\lambda,\tau)=e^{\mu_2(\lambda) \xi}Z(\lambda, \tau) \end{equation*}
yields a solution to the $A$ system which is necessarily in $W^u$, by the exponential decay condition as $\xi \rightarrow - \infty$.  Therefore $Z(\lambda,\tau)$ spans $W^u(\lambda,\tau)$ for each $\tau \in [-1,+1)$.  Lemma 6.1 in \cite{AGJ1990} tells us that the limit of $Z(\lambda,\tau)$ as $\xi \rightarrow \infty$ is non-zero and continuous in $\lambda$.  This means that $Z(\lambda,\tau)$ spans the unstable bundle for $\tau \in [-1,1]$, and has a non-singular projection on to $S^3$ for all $\tau$.\end{proof}

\begin{remark} The above Lemma \ref{lemma:nonsingproj} holds for systems with non-symmetric asymptotic limits provided the appropriate scaling is used.  The case of non-symmetric asymptotic limits will be treated in \S \ref{section:extension}, and we will return to this point in Proposition \ref{prop:nonsymmetric}. 
\end{remark}
\subsection{The induced phase on the Hopf bundle} 
\label{subsection:inducedphase}
Let $Z$ and $X^\pm(\lambda)$ be defined as in Lemma \ref{lemma:nonsingproj}, and $\hat{Z}$, $\hat{X}^\pm(\lambda)$ be their projections onto $S^3$, then $\hat{Z}$ defines a mapping to $S^3$ for which the following hold:
\begin{itemize}
\item $\hat{Z}(\lambda, \tau) \rightarrow \hat{X}^-(\lambda)$ as
 $\xi \rightarrow - \infty$
\item $\hat{Z}(\lambda ,\tau ) \rightarrow \zeta(\lambda) \hat{X}^+(\lambda)$ as $\xi \rightarrow +\infty$ 
 for some $\zeta(\lambda) \in \mathbb{C}$
\item $span_\mathbb{C}\{\hat{Z}(\lambda , \tau)\} \equiv W^u(\lambda, \tau)$
 \end{itemize} 
\begin{mydef}
Let $X^\pm(\lambda)$ be reference paths for $A_{\pm\infty}(\lambda)$ respectively.
The \textbf{induced phase}, with respect to $X^\pm(\lambda)$, is the complex scalar such that 
\begin{equation*}
\zeta(\lambda) \hat{X}^+(\lambda) \equiv \hat{Z}(\lambda,+1).
\end{equation*}
\end{mydef}
\begin{remark}Note that by our construction, both $\hat{Z}$ and $\hat{X}^+$ are unit vectors, ie: $\zeta(\lambda)\in S^1$.  In the simple case where $A_{-\infty}(\lambda)\equiv A_{+\infty}(\lambda)$ we may also take $X^+(\lambda) = X^-(\lambda)$ so that the induced phase is clearly a measure of the winding accumulated as the unstable manifold traverses $M$.  For systems with non-symmetric asymptotic limits we will need to adapt the method but the intuition remains the same.   
\end{remark}
Firstly we want to prove that, as a function of $s$, $\zeta$ is differentiable.  Having this condition, we will explore the connection between $\zeta(s)$, the choice of reference paths, the total multiplicity of the eigenvalue in $K^\circ$ and the geometric phase.  
\begin{prop}\label{prop:symmetric}
Let $X^\pm(\lambda)$ be non-degenerate reference paths for $A_{\pm\infty}(\lambda)$ respectively.  For each $\lambda \in K$, let us define $\zeta(\lambda)$ such that \mbox{$\hat{Z}(\lambda,+1)= \zeta(\lambda) \hat{X}^+(\lambda)$.}  We claim that if $\lambda(s)$ is a smooth parametrization of $K$, then
\begin{equation*}
\zeta\big(\lambda(s)\big):[0,1] \rightarrow S^1
\end{equation*}
 is a differentiable function.
\end{prop}
\begin{proof}
As in Lemma \ref{lemma:nonsingproj} the limit $Z(\lambda,\tau) \rightarrow Z(\lambda,+1)$ is non-zero for each $\lambda$ and $Z(\lambda,+1)$ is continuous.  Moreover, Lemma 3.7 in Alexander, Gardner \& Jones \cite{AGJ1990} tells us the convergence of the manifold $Z(\lambda,\tau)\rightarrow Z(\lambda,+1)$ is locally uniform outside of the spectrum of $\mathcal{L}$ and thus uniform on $K$.  Lemma 4.1 of \cite{AGJ1990} demonstrates that the solutions $Z(\lambda, \tau)$ are analytic in $\lambda$ for $\tau \in [-1,1)$.  But the limit of $Z(\lambda,\tau)$ converges uniformly for $\lambda\in K$, so the limiting function of $\lambda$, $Z(\lambda,+1)$, is also analytic in $\lambda$.  The spherical projection $\hat{Z}(\lambda,\tau)$ is not $\mathbb{C}$ analytic, but it will be real differentiable as a map from $\mathbb{R}^4\rightarrow S^3$.  This means the composition function $\hat{Z}\big(\lambda(s),+1\big)$ is differentiable with respect to the real parameter $s\in[0,1]$. The quantity $\zeta(\lambda)$ is given as the ratio of components of $\hat{Z}(\lambda,1)$ and $\hat{X}^+(\lambda)$ and is therefore differentiable in $s$.\end{proof}
\begin{mydef}
Let $Z$ and $X^\pm(\lambda)$ be defined as in Lemma \ref{lemma:nonsingproj} and fix some $\tau_0 \in [-1,1]$.  The \textbf{relative phase} of $Z(\lambda,\tau_0)$ is defined
\begin{align}
GP\Big(Z(K,\tau_0)\Big) - GP\Big(X^+(K)\Big)
\end{align}
\end{mydef}
\begin{lemma}\label{lemma:chern}
For non-degenerate reference paths $X^\pm(\lambda)$ for $A_{\pm\infty}(\lambda)$ and $Z,\hat{Z}$ as defined in Lemma \ref{lemma:nonsingproj} above, the \textbf{relative phase} of $\hat{Z}(\lambda,+1)$ equals the winding of the induced phase. 
\end{lemma}
\begin{proof}
 The natural connection on the Hopf bundle, $S^3$, is given by the 1-form
 \begin{equation*}\begin{matrix}
 \omega (V_p)  \equiv \langle V_p,p \rangle_{\mathbb{C}^2}, & & V_p \in T_p\left(S^3\right) \subset T_p\left(\mathbb{C}^2\right)
 \end{matrix} \end{equation*}
so that to calculate the geometric phase of $\hat{Z}\big(\lambda(s),+1\big)$, we consider
\begin{equation*} \begin{matrix}
 & \hat{Z}\big(\lambda(s),+1\big) &=& \zeta\big(\lambda(s)\big) \hat{X}^+\big(\lambda(s)\big) \\ \\
 \Rightarrow & \frac{d}{ds} \hat{Z}\big(\lambda(s),+1\big) &=& \zeta'\big(\lambda(s)\big)\lambda'(s) \hat{X}^+\big(\lambda(s)\big)+ \zeta\big(\lambda(s)\big) \frac{d}{ds}\hat{X}^+\big(\lambda(s)\big)\\ \\
 \Rightarrow & \omega\left(\frac{d}{ds} \hat{Z}\big(\lambda(s),+1\big)\right) &=& \overline{\zeta(s)} \zeta'\big(\lambda(s)\big) \lambda'(s) + \omega\left(\frac{d}{ds}\hat{X}^+\big(\lambda(s)\big)\right)
 \end{matrix} \end{equation*}
because $\hat{X}^+\big(\lambda(s)\big)$ is a unit vector and $\zeta(s)\in S^1$.  But the geometric phase of $\hat{Z}(\lambda,+1)$ is given by  
\begin{eqnarray}
 GP\Big(Z(K,+1)\Big)
  &=&
\displaystyle
 \frac{1}{2 \pi i} \int^1_0 \omega\left(\frac{d}{ds}\hat{Z}\big(\lambda(s),+1\big)\right) ds
\nonumber\\
  &=& \displaystyle
\frac{1}{2 \pi i} \int^1_0 \left[\overline{\zeta(s)} \zeta'\big(\lambda(s)\big) \lambda'(s) + \omega\left(\frac{d}{ds}\hat{X}^+\big(\lambda(s)\big)\right)\right] ds
\nonumber \\
  &=& \displaystyle
\frac{1}{2 \pi i} \int^1_0 \frac{\zeta'\big(\lambda(s)\big)}{\zeta\big(\lambda(s)\big)} \lambda'(s)ds + \frac{1}{2\pi i}\int^1_0\omega\left(\frac{d}{ds}\hat{X}^+\big(\lambda(s)\big)\right)  ds\\
  &=& \displaystyle
\frac{1}{2 \pi i} \int^1_0 \frac{\zeta'\big(\lambda(s)\big)}{\zeta\big(\lambda(s)\big)} \lambda'(s)ds + GP\Big(X^+(K)\Big)\,,
\label{theta-secondterm} 
  \end{eqnarray}  
  so that the relative phase of $Z(\lambda(s),+1)$ equals the winding of the induced phase.\end{proof}

To elaborate the dependence of the relative phase upon the reference paths
we introduce two lemmas.
\begin{lemma}\label{lemma:unique}
Given the contour $K$, let $V_1(\lambda)$ be a non-degenerate reference path and $V_2(\lambda)$ be a meromorphic reference path for $A_{+\infty}(\lambda)$.  Then
\begin{align}
GP\Big(V_1(K)\Big) = GP\Big(V_2(K)\Big) + Ind(V_2)
\end{align}
where $Ind(V_2)$ is plus or minus multiplicity of any zero or pole for $V_2$ in $K^\circ$.
\end{lemma}
\begin{proof}
Suppose $V_2$ has no essential singularity in $K^\circ$.  This is a generic choice as $V_2$ is an eigenvector of $A_{+\infty}(\lambda)$; $\lambda$ appears linearly in $\mathcal{L} - \lambda$ so that the only generic degeneracy of $V_2$ in $K^\circ$ is a pole or a zero. As eigenvectors, there must be some smooth scaling $\sigma:K \rightarrow \mathbb{C}^*$ such that $V_1(\lambda) \equiv \sigma(\lambda) V_2(\lambda)$.  Moreover, we can extend $\sigma(\lambda)$ over $K^\circ$ up to any zeros or poles enclosed by $K$. Consider the connection of $V_1\big(\lambda(s)\big)$, for some parametrization $\lambda(s)$,
\begin{align*}
\omega\left(\frac{d}{ds} \hat{V}_1\big(\lambda(s)\big)\right) & = \frac{d}{ds}\hat{\sigma}\big(\lambda(s)\big)\overline{\hat{\sigma}}\big(\lambda(s)\big) + \omega\left(\frac{d}{ds} \hat{V}_2\big(\lambda(s)\big)\right) 
\end{align*}
where $\hat{\sigma}\big(\lambda(s)\big)\equiv \frac{\sigma\big(\lambda(s)\big)}{\big| \sigma\big(\lambda(s)\big)\big|}$.  Therefore the geometric phase of $V_1$ equals that of $V_2$ plus the winding of $\hat{\sigma}\big(\lambda(s)\big)$; this agrees with $Ind(V_2)$ by the argument principle. 
\end{proof}
\begin{lemma}\label{lemma:initialcondition}
Let $V(\lambda)$ be a reference path for $A_{-\infty}(\lambda)$, with corresponding solution $V(\lambda,\tau)$, such that $V(\lambda)$ has a pole or zero in $K^\circ$.  Then the geometric phase of $V(\lambda,+1)$ equals the geometric phase of a solution evolved from a non-degenerate reference path plus the index of its degeneracy.
\end{lemma}
\begin{proof}
By definition $V(\lambda)$ is an eigenvector and therefore there must be some smooth scaling $\alpha:K \rightarrow \mathbb{C}^*$ and non-degenerate reference path $X^-(\lambda)$ such that  
\begin{align}
V^-(\lambda) \equiv \alpha(\lambda) X^-(\lambda)
\end{align}
Let $V$ and $Z$ denote solutions in the center unstable manifolds for these reference paths respectively, then by linearity of the flow the connection of the solution corresponding to $V(\lambda)$ is given
\begin{equation}\begin{matrix}
 &\hat{V}(\lambda,1) &=& \hat{\alpha}(\lambda) \hat{Z} (\lambda,1)\\
\Rightarrow&\omega\left(\frac{d}{ds} \hat{V}\big(\lambda(s)\big)\right) &=& \frac{d}{ds}\hat{\alpha}\big(\lambda(s)\big)\overline{\hat{\alpha}}\big(\lambda(s)\big) + \omega\left(\frac{d}{ds} \hat{Z}\big(\lambda(s)\big)\right) 
\end{matrix}\end{equation}
\end{proof}
\begin{cor}\label{cor:phase}
Given a choice of reference paths $X^{\pm}(\lambda)$ for $A_{\pm\infty}(\lambda)$, and $Z(\lambda,\tau)$ as defined above, the \textbf{relative phase} of $Z(\lambda,+1)$,
\begin{align}
GP\Big(Z(K,+1)\Big) - GP\Big(X^+(K)\Big),
\end{align}
equals the winding of the induced phase if and only if $X^\pm(\lambda)$ each have the same index of degeneracy.  In particular, the relative phase is the winding of the induced phase when $X^\pm(\lambda)$ are non-degenerate.
\end{cor}
\begin{proof}
This is a direct consequence of Lemmas \ref{lemma:chern}, \ref{lemma:unique} and \ref{lemma:initialcondition}.
\end{proof}

\subsection{The trivializations and the transition map} 
The unstable bundle is a non-trivial complex line bundle contained in the ambient trivial $\mathbb{C}^2$ vector bundle over the parameter sphere; for fixed $\lambda$, as $\tau$ moves between $\pm1$, the parameters in the sphere are the values $(\lambda,\tau)$ which describe the motion of solutions $Z(\lambda,\tau)$.  Taking a trivialization of this line bundle amounts to finding a linear isomorphism 
\begin{align*}
\phi_\alpha :\mathcal{U}_\alpha \times \mathbb{C}  \hookrightarrow & \mathcal{U}_\alpha \times \mathbb{C}^2
\end{align*}
where $\mathcal{U}_\alpha$ is a neighborhood in $M$, and the image of $\phi_\alpha$ is the unstable bundle over $\mathcal{U}_\alpha$.

\begin{mydef} Define the following:
\begin{itemize}

\item Let $H_-$ be the \textbf{lower hemisphere} of $M$, given by 
\begin{equation*}\label{eq:lhemisphere}
K^\circ \times \{\tau = -1 \} \cup K\times \{\tau \in [-1,1] \} \cup V \times \{ \tau =+1\}
\end{equation*}
where $V$ is an open neighborhood in $K^\circ$ homotopy equivalent to $S^1$ with $K$ in the closure of $V$.  Assume no eigenvalue of $\mathcal{L}$ is contained in $V$.  Thus $H_-$ is an open neighborhood of $M$.

\item Let $H_+$ be the \textbf{upper hemisphere} of $M$, given by 
\begin{equation*}\label{eq:uhemisphere}
K^\circ \times \{\tau = +1 \} \cup K\times \{\tau \in (-1,1] \}
\end{equation*}
so $H^+$ is an open neighborhood of $M$.

\item Let $Z$ and $\hat{Z}$ be as given in \S \ref{subsection:inducedphase}; abusing notation, let $Z$ and $\hat{Z}$ also denote their extensions into $V \times \{\tau = +1\}$ so that for $\lambda \in V$, $Z(\lambda, +1)$ is smoothly compatible with the values $Z(\lambda, +1)$, $\lambda \in K$.

\item For some non-degenerate reference path $X^+(\lambda)$ for $A_{+\infty}(\lambda)$, let $Y(\lambda,\tau)$ be in the center stable manifold of $X^+(\lambda)$.  Extend $Y$ into $K^\circ \times \{\tau = +1\}$ so that for $\lambda \in K^\circ$, $Y(\lambda, +1)$ is an eigenvector for the unstable direction of $A_{+\infty}(\lambda)$, smoothly compatible with the values on the boundary $K$.  We define the spherical projection of $Y$ to be $\hat{Y}$.
\end{itemize}
\end{mydef}
For fixed $(\lambda, \tau)$, where they are defined, $\hat{Z},\hat{Y}$ each span the unstable bundle.  $\hat{Z}$ is defined over $H_-$ and $\hat{Y}$ is defined over $H_+$, so that for any point $p$ in the unstable bundle we may choose a unique $z\in\mathbb{C}$ for which $p \equiv (\lambda, \tau, z \hat{Z})$ if $p$ is is over $H_-$, or choose a unique $y \in \mathbb{C}$ for which $p \equiv (\lambda , \tau, y \hat{Y})$ if $p$ is over $H_+$.
Thus the projections $\hat{Z},\hat{Y}$ give choices of trivializations for the unstable bundle over $H_-,H_+$ respectively. 
\begin{mydef}
Given $Z$, $Y$ as above, and a choice of hemispheres $H_{\pm}$, define the following maps:
\begin{equation*} \begin{matrix} \phi_- : & H_- \times \mathbb{C} & \hookrightarrow & H_- \times \mathbb{C}^2\\  \\
 & (\lambda, \tau, z) & \mapsto & \left(\lambda, \tau , z \hat{Z}(\lambda, \tau)\right)\\ \\
 \phi_+: & H_+\times \mathbb{C}& \hookrightarrow & H_+ \times \mathbb{C}^2 \\ \\
  & (\lambda , \tau , y) & \mapsto & \left(\lambda , \tau , y \hat{Y}(\lambda, \tau)\right) \\
  \end{matrix} \end{equation*}
These maps are the \textbf{trivializations of the unstable bundle} with respect to $H_\pm$, $\hat{Z}$ and $\hat{Y}$.  The maps $\phi_\pm$ are linear vector bundle isomorphisms, and their composition $\hat{\phi}\equiv\phi_+^{-1} \circ \phi_-$ defined on $H_- \cap H_+ \times \mathbb{C}$ is the \textbf{transition map} of the unstable bundle. 
\end{mydef}

Fixing $\tau$ such that $(\lambda,\tau) \in H_- \cap H_+$ $\forall \lambda \in K$, the transition map can be seen as a mapping from $S^1$ to $GL(1,\mathbb{C})$, ie: we take the restriction of the composition of trivializations to the $\lambda$ parameter
\begin{equation*} \begin{matrix}
{\phi_+}^{-1} \circ \phi_-(\lambda, \tau, -) : & K \cong S^1 & \rightarrow & GL(1,\mathbb{C})  \\ \\
 & (\lambda, -) & \mapsto & \hat{\phi}(-) \\ \\
 \hat{\phi}: & \mathbb{C} & \rightarrow & \mathbb{C}  \\ \\
 & z & \mapsto & y
  \end{matrix} \end{equation*}
where 
\begin{equation*}
z(\lambda,\tau)\hat{Z}(\lambda, \tau) = y(\lambda, \tau)\hat{Y}(\lambda, \tau)
\end{equation*}
Viewed this way 
\begin{equation*}
{\phi_+}^{-1} \circ \phi_-(-, \tau,-) \equiv \hat{\phi}_{\tau}
\end{equation*}
is seen to have a representation in the fundamental group of $GL(1,\mathbb{C})\cong \mathbb{C}^\ast$.  The fundamental group $\pi_1\left(\mathbb{C}^\ast\right) \cong \pi_1\left(S^1\right)\cong \mathbb{Z}$, so we identify $\left[\hat{\phi}_\tau\right] \cong d$ where $d\in \mathbb{Z}$ is the winding of $\hat{\phi}_\tau$ about $K$.
\begin{lemma}\label{lemma:eig}The winding of the map $\hat{\phi}_{\tau}(\lambda)$ is the \textbf{Chern number} of the unstable bundle, and equals the total multiplicity of the eigenvalues contained in $K^\circ$.
\end{lemma}
\begin{proof}
See Alexander, Gardner \& Jones \cite{AGJ1990}.
\end{proof}
\begin{lemma}\label{lemma:inducedphase}
For a choice of non-degenerate reference paths $X^\pm(\lambda)$ for $A_{\pm\infty}(\lambda)$, the winding of the induced phase equals the Chern number of the unstable bundle.
\end{lemma}
\begin{proof}
  Notice that for $(\lambda,+1)\in H_-\cap H_+$ the transition map can be described through the induced phase:
  \begin{equation*} \begin{matrix}
  z & \mapsto & z \hat{Z}\big(\lambda(s) , +1\big) & \equiv &z \zeta\big(\lambda(s)\big) \hat{X}^+\big(\lambda(s)\big)& \equiv &z \zeta\big(\lambda(s)\big) \hat{Y}\big(\lambda(s), +1\big)  & \mapsto & z \zeta\big(\lambda(s)\big) 
  
\end{matrix}  \end{equation*}
so that the transition map $\hat{\phi}$ is exactly given by $ z \mapsto \zeta\big(\lambda(s)\big) z$.
But the number of windings $\zeta\big(\lambda(s)\big)$ takes around the $K$ is given by 
\begin{align}
d = &  \frac{1}{2\pi i}\int_{\zeta(K)} \frac{1}{z} dz  \\
\label{eq:chernclass}  = &  \frac{1}{2\pi i} \int^1_0 \frac{\zeta'\big(\lambda(s)\big)}{ \zeta\big(\lambda(s)\big)} \lambda'(s) ds 
 \end{align}
so that the Chern number of the unstable bundle is given by the equation (\ref{eq:chernclass}) for the winding of the induced phase.\end{proof}   
 \subsection{The geometric phase and the transition map}
 We have now found the relationship between the induced phase $\zeta(s)$, for non-degenerate reference paths, and the Chern number of the unstable bundle over $M$.  However, we still need to interpret this in terms of the geometric phase in the Hopf bundle for a solution in the unstable manifold.  Let $Z,\hat{Z}$ be defined as in \S \ref{subsection:inducedphase}.  Then each $Z,\hat{Z}\in W^u$ for all $\xi$ and
 \begin{equation*} \begin{matrix}
 Z^\sharp &:=& e^{i \mu_2\big(\lambda(s)\big)\xi} Z
 \end{matrix}\end{equation*}
 is the corresponding solution to the $A$ system at $\xi$.  We want to show that the geometric phase of the two solutions agree for each $\xi$, and relate the phase to the winding of the transition map for the unstable bundle as $\xi \rightarrow +\infty$.  

\begin{theorem}[\textbf{The method of geometric phase---case I}]\label{theorem:casei}
Given a choice of reference paths $X^\pm(\lambda)$ for $A_{\pm\infty}(\lambda)$ and $Z$ defined as in \S \ref{subsection:inducedphase}, the asymptotic relative phase of $Z\big(\lambda,\tau(\xi)\big)$,
\begin{align}
\lim_{\xi\rightarrow\infty}GP\Big(Z\big(K,\tau(\xi)\big)\big) - GP\Big(X^+(K)\Big),
\end{align}
equals the total multiplicity of the eigenvalues enclosed by $K$ if $X^\pm(\lambda)$ are non-degenerate.
\end{theorem}
\begin{proof}
This theorem is a direct consequence of Lemmas \ref{lemma:eig} and \ref{lemma:inducedphase}, and Corollary \ref{cor:phase}.\end{proof}

Finally we will explore the relationship between the solutions to the $B$ system where we calculate the phase in the proof, and the solutions to the $A$ system.  

\begin{prop}\label{prop:ABequiv}
Let $X^-(\lambda)$ be a reference path for $A_{\pm\infty}$ and suppose $Z$ and $Z^\sharp$ are solutions to the $B$ and $A$ system respectively, and that they agree at $\xi_0$; then for arbitrary finite $\xi$ the geometric phase of $Z\big(\lambda,\tau(\xi)\big)$ and $Z^\sharp\big(\lambda,\tau(\xi)\big)$ agree. 
\end{prop}
\begin{proof}
Suppose $\mu_2(\lambda)\equiv \alpha(\lambda) + i \beta(\lambda)$, and recall the solution to the $A$ system given by 
\begin{equation*}
Z^\sharp\big(\lambda,\tau(\xi)\big) = e^{\mu_2(\lambda)(\xi-\xi_0)}Z\big(\lambda,\tau(\xi)\big).
\end{equation*}  
Without loss of generality, we will suppose $\xi_0 = 0$ so that $Z^\sharp$ is the unique solution to the $A$ system that agrees with $Z$ at $\xi=0$; the proof will not depend on the constant.  The projection of $Z^\sharp$ onto the Hopf bundle is given by \begin{equation*}
\hat{Z}^\sharp\big(\lambda,\tau(\xi)\big) \equiv e^{i\beta(\lambda)\xi}\hat{Z}\big(\lambda,\tau(\xi)\big),
\end{equation*} 
so that calculating the phase:
\begin{equation}\label{eq:ABequiv} \begin{matrix}
 & \hat{Z}^\sharp\big(\lambda(s),\tau(\xi)\big) &=& e^{i\beta\big(\lambda(s)\big)\xi}\hat{Z}\big(\lambda(s),\tau(\xi)\big) \\ \\
 \Rightarrow & \frac{d}{ds} \hat{Z}^\sharp\big(\lambda(s),\tau(\xi)\big) &=& i\beta '\big(\lambda(s)\big)\lambda'(s) \xi e^{i\beta\big(\lambda(s)\big)\xi}\hat{Z}\big(\lambda(s),\tau(\xi)\big) + e^{i\beta\big(\lambda(s)\big)\xi} \frac{d}{ds} \hat{Z}\big(\lambda(s),\tau(\xi)\big)\\ \\
 \Rightarrow & \omega\left(\frac{d}{ds} \hat{Z}^\sharp\big(\lambda(s),\tau(\xi)\big)\right) &=&i\beta ' \big(\lambda(s)\big)\lambda'(s)\xi + \omega\left(\frac{d}{ds} \hat{Z}\big(\lambda(s),\tau( \xi)\big) \right)
 \end{matrix} \end{equation}
But $\mu_2(\lambda),\mu_2'(\lambda)$ are each holomorphic by construction so that 
\begin{equation*}
\int_K \mu_2'(\lambda)= \int_K \alpha'(\lambda)+i \int_K\beta'(\lambda) \equiv 0
\end{equation*}
and the real and imaginary parts both must equal zero. The $i\beta'\big(\lambda(s)\big)\lambda'(s)\xi$ term thus vanishes in equation (\ref{eq:ABequiv}) when integrated for $s\in[0,1]$.  This proves the geometric phase of $Z^\sharp(\lambda,\xi)$ of the $A$ system corresponds to the phase of the solution $Z(\lambda,\xi)$ for the $B$ system for arbitrary $\xi$.  For systems defined on $\mathbb{C}^2$ we may thus obtain the total multiplicity of the eigenvalues contained in $K^\circ$ with a solution to either the $A$ or $B$ system utilizing the method of geometric phase.\end{proof}

\section{Extending the two dimensional method}
\label{section:extension}
In this section we adapt the techniques developed for systems on $\mathbb{C}^2$ to take advantage of the full generality in which the unstable bundle can be constructed.  Firstly we will extend the techniques to the case when there are $k>1$ unstable directions, and once we have a general method, we will consider systems with non-symmetric asymptotic limits.
\subsection{The $n-$dimensional case and determinant unstable bundle}

Suppose now the operator $\mathcal{L}$ defines an $A$ system and $B$ system on $\mathbb{C}^n$.  If for all $\lambda \in \Omega$, $A_\infty = A_{\pm\infty}$ has one unstable direction and $n-1$ stable directions, the proof in two dimensions holds; although the ambient complex dimension has increased, the unstable bundle is still one dimensional.  Likewise if the stable manifold is $1$-dimensional, we may calculate the Chern number of the analogous stable bundle without any serious modification of the method.

Suppose more generally there are $1< k < n-1$ unstable directions for the system $A_\infty$.  The $k$ dimensional unstable bundle is again formed from the unstable manifold $W^u(\lambda,\tau)$ of the critical point $(0,-1)$, and the Chern number of this vector bundle equals the total multiplicity of the eigenvalues contained in $K^\circ$.  However, it is no longer sufficient to consider solutions only corresponding to a single eigenvector, as this will not capture the information of the full unstable bundle.  To write the transition map of the unstable bundle, $E$, as a value in $S^1$ we must introduce the determinant bundle constructed from a $k$ dimensional vector bundle.  This technique uses subspace coordinates, reducing the dimension of the unstable bundle to one, while raising the ambient complex dimension of the system.  With respect to this coordinatization, the unstable manifold is a trajectory on which we can again calculate the geometric phase, and the goal is thus to apply the same method we use in $\mathbb{C}^2$ to the determinant bundle of the $k$ dimensional unstable space.  

\begin{mydef}
The $\mathbf{k^{\text{th}}}$ \textbf{exterior power of} $\mathbf{\mathbb{C}^n}$, $\Lambda^k\left(\mathbb{C}^n\right) \equiv \mathbb{C}^{n \choose k}$, is the complex vector space of non-degenerate $k$ forms on $\mathbb{C}^n$. $\Lambda^k\left(\mathbb{C}^n\right)$ is spanned by
\begin{equation*}\begin{matrix}
v = v_1 \wedge \cdots \wedge v_k & & v_i \in \mathbb{C}^n\hspace{2mm} \forall i
\end{matrix}\end{equation*}
and $v$ is non-degenerate provided $\{v_i\}_1^k$ are linearly independent in $\mathbb{C}^n$. 
\end{mydef}

\begin{mydef}
Given a dynamical system 
\begin{equation*} \begin{matrix}
X'=AX & &  ' = \frac{d}{d\xi}
 & & X\in \mathbb{C}^n
\end{matrix}\end{equation*} 
 let  $Y = Y_1 \wedge \cdots \wedge Y_k\in \Lambda^k\left(\mathbb{C}^n\right)$.  The associated $\mathbf{A^{(k)}}$\textbf{ system} on $\Lambda^k\left(\mathbb{C}^n\right)$ is generated by
\begin{align} 
 Y' &=A^{(k)} Y \\ 
    &:= AY_1\wedge \cdots \wedge Y_k + \cdots + Y_1\wedge \cdots \wedge AY_k  
\label{eq:A^k}
\end{align}
\end{mydef}
\begin{remark}
By equation (\ref{eq:A^k}) it is clear that the eigenvalues for the $A^{(k)}$ system are the sums of all $k$-tuples of eigenvalues for $A$.  Thus for $A^{(k)}$, there is a unique eigenvalue of largest positive real part given by the sum of all eigenvalues with positive real part, including multiplicity.  
\end{remark}

\begin{mydef}
Suppose $\mathcal{L}$ defines a system of the form (\ref{eq:symmetric}) on $\mathbb{C}^n$.  Denote $\{ \mu^\pm_1, \cdots , \mu^\pm_k \} $ the eigenvalues of positive real part for $A_{\pm\infty}(\lambda)$ respectively, and define $\mu^\pm:= \sum^k_{i=1} \mu^\pm_i $.  
The corresponding $\mathbf{A^{(k)}}$\textbf{ and }$\mathbf{B^{(k)}}$\textbf{ systems} on $\Lambda^k\left(\mathbb{C}^n\right)\equiv \mathbb{C}^{n\choose k}$ are defined:
\begin{equation}\label{eq:A^ksystem} \begin{matrix}
Y'=A^{(k)}(\lambda, \tau) Y & A^{(k)}_{\pm\infty}(\lambda) = \lim_{\xi \rightarrow \pm\infty} A^{(k)}(\lambda,\tau) \\ \\
\tau' = \kappa(1-\tau^2) & \\ \\
\end{matrix} \end{equation}

\begin{equation}\label{eq:B^ksystem} \begin{matrix}
B^{(k)}(\lambda, \tau):= \left(A^{(k)}(\lambda,\tau)- \mu^-(\lambda)\right) & X'=B^{(k)}X \\ \\
  B^{(k)}_{\pm\infty} (\lambda) := \lim_{\xi \rightarrow \pm \infty} B^{(k)} (\lambda, \tau) & \tau' = \kappa(1-\tau^2) \\ \\
\end{matrix} \end{equation}
\end{mydef}
Allen \& Bridges \cite{allenbridges2002} demonstrate that there is an explicit algorithm to compute the $A^{(k)}$ system (\ref{eq:A^k}) on the exterior power $\Lambda^k\left(\mathbb{C}^n\right)$ where the coefficients of $A^{(k)}$ are calculated through the inner product on $\mathbb{C}^n$.  For the symmetric form of system (\ref{eq:B^ksystem}), $B^{(k)}_\infty$ has a center direction of critical points, an unstable real direction, and all other directions are stable; the line of critical points is given by the span of the wedge of linearly independent eigenvectors corresponding to $\{ \mu^-_1, \cdots , \mu^-_k\}$.

\begin{mydef}\label{def:detbundle}
For all $(\lambda,\tau) \in M$, let $\{w_i(\lambda,\tau)\}^k_1$ be a spanning set for the unstable manifold $W^u$ at $(\lambda,\tau)$, and define
\begin{equation*}
\Lambda^k\big(W^u(\lambda,\tau)\big) \equiv span_{\mathbb{C}}\{w_1(\lambda,\tau) \wedge \cdots \wedge w_k(\lambda,\tau) \}\,.
\end{equation*}
Then $\Lambda^k\big(W^u(\lambda,\tau)\big)$ can be taken as the fiber for a non-trivial
vector bundle $\Lambda^k(E)$ over $M$ with projection
$\pi_{E^k}:E^k\to M$,
\begin{equation}\begin{CD}
\Lambda^k\big(W^u(\lambda,\tau)\big) @>>> \Lambda^k(E)\\
& & @VVV \pi_{E^k} \\
& & M
\end{CD}\end{equation}  
$\Lambda^k(E)$ is called the \textbf{determinant bundle of the unstable manifold} over $M$; henceforth we will refer to $\Lambda^k(E)$ simply as the \textbf{determinant bundle}.  The determinant bundle is a \textbf{line bundle}.
\end{mydef}
Let the transition map of the unstable bundle $E$ be denoted $\hat{\phi}_E$.  The determinant bundle acquires its namesake from the construction of its transition map $\hat{\phi}^k_E$.  The transition map of the $k$-dimensional unstable bundle is a $\lambda$ dependent, non-singular mapping of $k$-frames of $n$ dimensional complex vectors.  Restricting to the equator of $M$, we thus interpret the transition map
\begin{equation*} \begin{matrix}
\hat{\phi}_E: & S^1 &\rightarrow& GL(\mathbb{C},k) \\ \\
 & \lambda &\mapsto & \psi(\lambda) \\ \\
\end{matrix} \end{equation*}
so that it defines an element of $\pi_1\big(GL(\mathbb{C},k)\big)$.  But notice, $\det\left(\hat{\phi}_E(\lambda)\right) \in GL(\mathbb{C},1)$ for all $\lambda \in K$, so that the determinant induces a homomorphism of fundamental groups
\begin{equation*} \begin{matrix}
\det_* : & \pi_1\big(GL(\mathbb{C},k)\big) & \rightarrow & \pi\big(GL(\mathbb{C},1)\big) \\ \\
	& \left[\hat{\phi}_E\right] & \mapsto & \left[\det\circ \hat{\phi}_E\right] \\
\end{matrix} \end{equation*}
\begin{mydef}
The mapping,  
\begin{align}
\det \circ \hat{\phi}_E(\lambda) \equiv  \hat{\phi}^k_E,
\end{align}
is the \textbf{transition map of the determinant bundle}.  
\end{mydef}
\begin{lemma}
The Chern number of the determinant bundle of the unstable manifold over $M$ equals the Chern number of the unstable bundle, and therefore the total multiplicity of eigenvalues for $\mathcal{L}$ contained in $K$.
\end{lemma}
\begin{proof}
This is proven by Alexander, Gardner \& Jones \cite{AGJ1990}
\end{proof}
\subsection{Solving the $k$ unstable direction case}
For systems (\ref{eq:A^ksystem}) with symmetric asymptotic limits, we may utilize the method of geometric phase, calculating the geometric phase of the solution $Z(\lambda,\tau)$ corresponding to the eigenvalue of most positive real part, where $Z(\lambda,\tau)$ describes the determinant bundle.  These modifications are presented in the following theorem.
\begin{theorem}[\textbf{The method of geometric phase---case II}]\label{theorem:caseii}
Let the $A^{(k)}$ and $B^{(k)}$ systems be defined as in equations (\ref{eq:A^ksystem}) and (\ref{eq:B^ksystem}) above.  Let $X^\pm(\lambda)$ be reference paths for $A^{(k)}_{\pm\infty}(\lambda)$ and suppose $Z\big(\lambda,\tau(\xi)\big)$ is in the center-unstable manifold of $X^-(\lambda)$ with respect to $B^{(k)}$.  Then the asymptotic relative phase of $Z\big(\lambda,
\tau(\xi)\big)$,
\begin{align}
\lim_{\xi\rightarrow\infty}GP\Big(Z\big(K,\tau(\xi)\big)\Big) - GP\Big(X^+(K)\Big),
\end{align}
 equals the total multiplicity of the eigenvalues enclosed by the contour $K$ if $X^\pm(\lambda)$ are non-degenerate. 
\end{theorem}
\begin{proof}
As in the two dimensional case, $Z$ forms a $\mathbb{C}$ analytic section of the line bundle over $M$ for $\tau \in[-1,1]$.  \S 4 of Alexander, Gardner \& Jones \cite{AGJ1990} shows that this solution is analytic on $[-1,+1)$ and \S 6 shows that the limit as $\xi\rightarrow +\infty$ is non-zero and continuous.  The proof of locally uniform convergence in Proposition \ref{prop:symmetric} holds here as well, so that the extension of $Z$ to $Z(\lambda,+1)$ is $\mathbb{C}$ analytic.

Therefore we take the projection of $Z$, $\hat{Z}$, onto the sphere 
\begin{equation*}
S^{\left(2 {n \choose k} -1\right)} \subset \mathbb{C}^{n \choose k} \cong \Lambda^k\left(\mathbb{C}^n\right),
\end{equation*}
and with respect to $X^\pm(\lambda)$ we again obtain the induced phase $\zeta(\lambda)$.

Let $Y(\lambda,\tau)$ be a solution to the $B^{(k)}$ system that is in the center-stable manifold of a non-degenerate reference path $\hat{X}^+(\lambda)$ at $\tau=+1$, and let $\hat{Y}$ be the projection of this solution.  The trivializations of the determinant bundle can be expressed in terms of $\hat{Z}$ and $\hat{Y}$, which yields transition map 
\begin{equation*}\hat{Z}(\lambda,+1) \equiv \zeta(\lambda) \hat{X}^+(\lambda) \equiv \zeta(\lambda) \hat{Y}(\lambda,+1)
\end{equation*}
The winding of $\zeta(\lambda)$ is thus equal to the Chern number of the determinant bundle, and is related to the geometric phase of $Z(\lambda,+1)$ by the same formulation described in the two dimensional case.\end{proof}

Thus in the case of $k$ unstable directions, we may calculate the total multiplicity of the eigenvalues contained in the region $K^\circ$ by an adaptation of the method of geometric phase applied to the determinant bundle of the unstable manifold.  The same proof as in Lemma \ref{prop:ABequiv} will demonstrate that the geometric phase is equivalent in both the $A^{(k)}$ and $B^{(k)}$ systems.

\subsection{The general case for systems on unbounded domains}
In the preceding sections we developed a method for finding the total multiplicity of eigenvalues for $\mathcal{L}$ in the region $K^\circ$, but the method was restricted to the case for which $\lim_{\xi\rightarrow -\infty} A(\lambda,\xi) \equiv \lim_{\xi \rightarrow +\infty} A(\lambda,\xi)$.  The unstable bundle construction, however, is valid for general systems $A_{\pm \infty}$ that \textbf{split} in $\Omega$, ie: each have exactly $k$ unstable, and $n-k$ stable directions for every $\lambda \in \Omega$. The final modification we will make is to account for systems with non-symmetric asymptotic limits.  The following construction will reduce to that in the previous sections if the system is symmetric or the dimension of the unstable manifold is $k=1$, so this may be considered the fully general statement of the method of geometric phase for systems on unbounded domains.

Suppose we have the determinant bundle system
\begin{equation}\label{eq:A^ksystemnonsymmetric} \begin{matrix}
Y'=A^{(k)}(\lambda, \tau) Y & A^{(k)}_{\pm \infty}(\lambda) = \lim_{\xi \rightarrow \pm\infty} A^{(k)}(\lambda,\tau) \\ \\
\tau' = \kappa(1-\tau^2) & \\ \\
\end{matrix} \end{equation}
derived from the flow $Y'=AY$ on $\mathbb{C}^n$.

Given a non-degenerate reference path for $A_{-\infty}(\lambda)$, $X^-(\lambda)$, we may construct the center-unstable manifold of the direction of critical points at $\tau = -1$ in the $B^{(k)}$ system as before. However, the behavior of such a solution will differ when $\tau \rightarrow +1$.  The dominating unstable eigenvalue for the system at $\tau=+1$ does not in general equal the value at $\tau = -1$, but we must guarantee the hyper-spherical projection is non-singular as $\xi \rightarrow \infty$.

\begin{mydef} \label{mydef:gamma}
Let $\mu^\pm(\lambda)$ be the eigenvalue of most positive real part for $A^{(k)}_{\pm\infty}(\lambda)$.  For a reference path $X^-(\lambda)$ for $A^{(k)}_{-\infty}(\lambda)$ define the center-unstable manifold of $X^-(\lambda)$ in the $B^{(k)}$ system to be $Z(\lambda,\tau)$ for $\tau\in[-1,1)$.  Define
\begin{equation*} \begin{matrix}
Z^\sharp(\lambda,\tau) :=  e^{\left(\mu^-(\lambda)\xi\right)}Z(\lambda,\tau) & & \tau \in (-1,+1)
\end{matrix} \end{equation*}
so that
\begin{equation} \begin{matrix}
\Gamma(\lambda,\tau) &:=& \begin{cases}
e^{\left(-\mu^- \xi\right)}Z^\sharp(\lambda,\tau) & \text{for $\tau \in [-1,0)$} \\
e^{\left(-\mu^+ \xi\right)}Z^\sharp(\lambda,\tau) & \text{for $\tau \in [0,+1)$} \\
\lim_{\xi\rightarrow \infty} e^{\left(-\mu^+ \xi\right)}Z^\sharp(\lambda,\tau) & \text{for $\tau =+1$}\\
\end{cases} \end{matrix} \end{equation}
\end{mydef}
\begin{prop}\label{prop:nonsymmetric}
$\Gamma(\lambda,\tau)$ satisfies the equation
\begin{equation}\begin{matrix}\label{eq:specialsystem}
Y' = \Psi(\lambda,\xi)Y  & &
\Psi &=& \begin{cases}
\big(A^{(k)}(\lambda,\xi) - \mu^-(\lambda)I\big)\text{ for $\xi\in(-\infty,0)$} \\
\big(A^{(k)}(\lambda,\xi) - \mu^+(\lambda)I\big)\text{ for $\xi\in[0,+\infty)$}
\end{cases}
\end{matrix}\end{equation}
Moreover, $\Gamma(\lambda,\tau)$ is non-zero and analytic in $\lambda$ for fixed $\tau$, and spans the determinant bundle $\forall (\lambda,\tau)\in H_-$.
\end{prop}
\begin{proof}
Notice that $\Gamma(\lambda,\tau)$ is a solution to equation (\ref{eq:specialsystem}) is by construction, and moreover, the analyticity of $\Gamma$ for $\tau \in [-1,+1)$ is obvious from the analyticity of $Z$.  Under the flow defined by 
\begin{align}
Y'=\big(A^{(k)}(\lambda,\tau) - \mu^+(\lambda)I\big)Y
\end{align} 
the eigenvector corresponding to $\mu^+(\lambda)$ is once again a line of critical points; we want the solution $\Gamma(\lambda,\tau)$ to converge uniformly in $\lambda$ to a non-zero critical point defined this way.  This uniform convergence is obtained by using Lemma 6.1 in Alexander, Gardner \& Jones \cite{AGJ1990} as we did in our Lemma \ref{lemma:nonsingproj}, but in full generality.  Likewise, following the proof of our Proposition \ref{prop:symmetric}, $\Gamma(\lambda,\tau)$ indeed defines a section of the determinant bundle over the lower hemisphere $H_-$.\end{proof}
\begin{theorem}[\textbf{The method of geometric phase---general unbounded systems}]\label{theorem:caseunbounded}
If $X^\pm(\lambda)$ are reference paths for $A^{(k)}_{\pm\infty}(\lambda)$, and $\Gamma\big(\lambda,\tau(\xi)\big)$ is defined as in Definition \ref{mydef:gamma}, then the asymptotic relative phase of $\Gamma\big(\lambda,\tau(\xi)\big)$,
\begin{align}
\lim_{\xi\rightarrow\infty}GP\Big(\Gamma\big(K,\tau(\xi)\big)\Big) - GP\Big(X^+(K)\Big),
\end{align}
 equals the total multiplicity of the eigenvalues enclosed by the contour $K$ if $X^\pm(\lambda)$ are non-degenerate. 
\end{theorem}
\begin{proof}
To adapt the determinant bundle method from here, we need only define $Y, \hat{Y}$ appropriately so they converge to a non-degenerate reference path for $A^{(k)}_{+\infty}$.  The construction of the induced parallel translation will follow analogously, as will the lemmas of \S \ref{section:twodim}.\end{proof}
\begin{remark}
The equivalence of the geometric phase for $\Gamma(\lambda,\tau)$ and $Z^\sharp(\lambda,\tau)$ for $\tau \in (-1,1)$ follows from the proof of Proposition \ref{prop:ABequiv}. 
\end{remark}
\section{Boundary value problems}
\label{section:boundary}
Gardner \& Jones further developed the bundle construction for the Evans Function to study boundary value problems with parabolic boundary conditions \cite{GJ91}, ie: problems of the form

\begin{equation*} \begin{matrix}

 u_t = Du_{xx} + f(x,u,u_x) & & (0<x<1) & \\ \\
u(x,0)=u_0 & B_0 u = 0 & B_1 u = 0\\

\end{matrix} \end{equation*}

where $u \in \mathbb{R}^n$, $f:\mathbb{R}^{2n+1} \rightarrow \mathbb{R}^n$ is $C^2$.  The matrix $D$ is a positive diagonal matrix and the boundary operators are defined
\begin{align*}
B_0 u & =  D^0 u(0,t) + N^0 u_x(0,t) \\
B_1 u & =  D^1 u(0,t) + N^1 u_x(0,t) 
\end{align*}
such that $D^j,N^j$ are diagonal with entries $\alpha^j_i,\beta^j_i$ respectively that satisfy

\begin{equation*} \begin{matrix}
{\left(\alpha^j_i\right)}^2 + {\left(\beta^j_i\right)}^2 = 1 &  &1 \leq i \leq n; & i =1,2\\
\end{matrix} \end{equation*}

Austin \& Bridges built upon and generalized these bundle methods into a vector bundle construction for boundary value problems for which the boundary conditions can depend on $\lambda$, and allow for general splitting of the boundary conditions \cite{BRI03}.  In this section we will consider how the method of geometric phase can be adapted to boundary value problems, using the techniques Austin and Bridges developed for the general boundary conditions.

\subsection{Constructing the boundary bundle for $\mathbb{C}^n$}

Suppose for $n\geq 2$ we are given a system of ODE's defining a flow on $\mathbb{C}^n$, derived from the linearization $\mathcal{L}$ of a reaction diffusion equation about a steady state.  Assume the system is of the form
\begin{equation}\label{eq:boundaryprob} \begin{matrix}
u_x = A(\lambda,x)u  & 0<x<1 & \lambda \in \Omega \subset \mathbb{C} \\ \\
a_i^*\big(\bar{\lambda}\big): \mathbb{C} \rightarrow \mathbb{C}^n & i = 1, ... ,n-k &
b_i^*\big(\bar{\lambda}\big): \mathbb{C} \rightarrow \mathbb{C}^n & i = 1, ... , k \\
\end{matrix} \end{equation}
where $A(\lambda,x)$ depends analytically on $\lambda$, and the $a_i^*,b_i^*$ are holomorphic functions of $\bar{\lambda}$ that describe the boundary conditions for the operator $\mathcal{L}$---the specific conditions are described with respect to the \textbf{section product} below.

The ambient trivial bundle is once again constructed from the product $M \times \mathbb{C}^n$.  The vectors $\left(\lambda,x,a_i^*\right), \left(\lambda,x,b_i^*\right)$ for each $(\lambda, x)\in M$ are anti-holomorphic sections of the trivial bundle, motivating the above dual notation.  
\begin{mydef}
For a pair $\nu(\lambda,x)$, $\eta(\lambda,x)$ where $\nu$ is a holomorphic section and $\eta$ is an anti-holomorphic section of the trivial bundle $M\times \mathbb{C}^n$, their product is defined as:

\begin{equation}\label{eq:sectionproduct} 
\langle \eta , \nu \rangle_\lambda  = \sum_{j=1}^n \overline{\eta_j\big(\overline{\lambda}\big)}  \nu_j(\lambda)
\end{equation}
where $\eta_j, \nu_j$ are their respective components.
\end{mydef}
\begin{remark}
This scalar product is holomorphic for all $\lambda \in \Omega$, and the boundary value problem is formulated as follows: $u(\lambda,x)$ is an eigen function of the operator $\mathcal{L}$ for the eigenvalue $\lambda$ if and only if $u(\lambda,x)$ is a solution to $u_x = A(\lambda,x)u$ and 
\begin{equation*} \begin{matrix}

\left\langle a_i^*\big(\bar{\lambda}\big), u(\lambda,0) \right\rangle_\lambda = 0 & & i = 1, ..., n-k \\ \\
\left\langle b_i^*\big(\bar{\lambda}\big), u(\lambda,1) \right\rangle_\lambda = 0 & & i = 1, ..., k \\
\end{matrix} \end{equation*}
\end{remark}
A significant difference in this construction from the unbounded systems is that there are no dynamics to consider on the caps of the parameter sphere, and we will not be concerned with the eigenvalues of a limiting system.  What is needed then is an analogue to the unstable bundle that will trace the dynamics and pick up winding while traversing the parameter sphere between $\tau = \mp 1$. One choice is the orthogonal compliment to the initial conditions, dimension $k$, and the manifold defined by their evolution.

We need to show these subspaces, and their images under the flow, vary holomoprhically with respect to $\lambda \in \Omega$.  With a holomorphic basis, we may construct a non-trivial vector bundle over M through which we can calculate the geometric phase with the determinant bundle.  
\begin{theorem}
For a system of the form (\ref{eq:boundaryprob}) derived from the operator $\mathcal{L}$ there exists analytic choices of orthogonal bases for $\mathbb{C}^n$ such that
\begin{align}
&V_0 :=\{\nu_i(\lambda):\lambda\in\Omega \}_1^{n-k}  &U_0 := \{\xi_i(\lambda) : \lambda \in \Omega \}_1^k & & V_0 \oplus U_0 = \mathbb{C}^n\label{eq:rpath1} \\ 
&V_1 := \{\upsilon_i(\lambda):\lambda\in\Omega \}_1^{n-k}  & U_1:=\{\eta_i(\lambda) : \lambda \in \Omega \}_1^k & & V_1 \oplus U_1 = \mathbb{C}^n \label{eq:rpath2} 
\\
&span_{\mathbb{C}}\{\nu_i\}_1^{n-k} = span_{\mathbb{C}}\{a_i^*\}_1^{n-k} &  span_{\mathbb{C}}\{\eta_i\}_1^{k} = span_{\mathbb{C}}\{b_i^*\}_1^{k}
\end{align}
and with respect to the product of sections (\ref{eq:sectionproduct})
\begin{align}
\left\langle a_i^*\big(\bar{\lambda}\big) , \xi_j(\lambda) \right\rangle_\lambda =& 0 & 0 \leq i \leq n-k,\hspace{2mm} 0\leq j\leq k 
\\ 
\left\langle b_i^*\big(\bar{\lambda}\big) , \upsilon_j(\lambda) \right\rangle_\lambda =& 0 & 0 \leq i \leq k,\hspace{2mm} 0\leq j\leq n-k  
\end{align}
\end{theorem}
\begin{proof}
This is the content of Austin \& Bridges' results in Lemmas 3.1 through 3.3 in \cite{BRI03} and the reader is referred there for a full discussion.\end{proof}
\begin{remark}
Reformulating the problem in this context, $u(\lambda,x)$ is an eigenfunction of $\mathcal{L}$ with eigenvalue $\lambda$ if and only if  
\begin{align*}
u(\lambda,0)\in span_\mathbb{C}\{U_0(\lambda)\}\\
u(\lambda,1) \in span_\mathbb{C} \{V_1(\lambda)\}.
\end{align*}
\end{remark}
Thus we will define a ``boundary bundle'' over $M$ by foliating the subspaces $U_0(\lambda),U_1(\lambda)\subset \mathbb{C}^n$ on the caps of M, and constructing subspaces that connect $U_0(\lambda),U_1(\lambda)$ while picking pick up the information of the flow.  We choose $U_0(\lambda),U_1(\lambda)$ as fibers above the caps of the boundary bundle because if $\lambda$ is not an eigenvalue, a solution to the system $u'=Au$ cannot be in the span $U_0$ at $x=0$ and in the span of $V_1$ at $x=1$.  Any collection of solutions $\{\gamma_1(\lambda,x),...,\gamma_k(\lambda,x)\}$ that satisfy the boundary conditions at $x=0$, and are linearly independent for $(\lambda,0)$, will be linearly independent for $(\lambda,x)$ where $x\in[0,1)$. In particular when $\lambda$ is not an eigenvalue of $\mathcal{L}$, then $\{\gamma_1(\lambda,1),...,\gamma_k(\lambda,1)\}$ are linearly independent and must span $some$ compliment of $V_1(\lambda)$; in general this need not be the orthogonal compliment, ie: $U_1(\lambda)$, but it is possible to smoothly deform the solutions into $U_1(\lambda)$ with the projection operator. 
\begin{mydef}
Define the $\lambda$ dependent projection operator
\begin{equation*}
Q_\lambda: \mathbb{C}^n  \rightarrow  U_1(\lambda) 
\end{equation*}
and define the orthogonal projection operator
\begin{equation*}
P_\lambda = (I - Q_\lambda): \mathbb{C}^n  \rightarrow  V_1(\lambda)
\end{equation*}
\end{mydef}
\begin{prop}\label{prop:boundarybundle}
Let $u_i(\lambda,x)$ be solutions to the flow on $\mathbb{C}^n$ such that \mbox{$u_i(\lambda,0)=\xi_i(\lambda)$} for each $i=1,...,k$, and let $\{\eta_i(\lambda):i=1,...,k \text{ and } \lambda\in\Omega\}$ be a holomorphic basis for $U_1(\lambda)$. Define
\begin{equation*} \begin{matrix}
\sigma_i(\lambda,x) & \equiv &
\begin{cases}
(\text{I} -xP_\lambda)\big(u_i(\lambda,x)\big)& (\lambda,x)\in K \times [0,1]\\ \\
 \xi_i(\lambda) & (\lambda,0) \in K^\circ\times\{0\} \end{cases}
\end{matrix} \end{equation*}  
Then $\{\sigma_j(\lambda,x)\}$ are linearly independent and holomorphic for all $(\lambda,x) \in M \setminus \left(K^\circ \times \{1\}\right)$.
\end{prop}
\begin{proof}
This proposition follows immediately from the results of \S 4 in Austin \& Bridges \cite{BRI03}, and we may now use the above construction to describe the boundary bundle over $M$.\end{proof}
\begin{mydef}
Define $\mathcal{E}_{\lambda,x}\subset \mathbb{C}^n$ to be the $k$ dimensional subspace spanned by $\{\sigma_i(\lambda,x) : i = 1, \cdots, k\}$ for $(\lambda,x)\in K \times [0,1]$.   Over $K^\circ \times \{0\}$ define $\mathcal{E}_{\lambda,0} =span_\mathbb{C}\{U_0(\lambda)
\}$.  Finally over $K^\circ \times \{1\} $ we define $\mathcal{E}_{\lambda,1} =span_\mathbb{C}\{U_1(\lambda)\}$. For the fibers defined as above, $E \equiv \{(\lambda,x, \mathcal{E}_{\lambda,x}) : (\lambda,x)\in M\}$ is defined to be the \textbf{boundary bundle} with respect to equation (\ref{eq:boundaryprob}) over $M$. 
\end{mydef}
It follows from Proposition \ref{prop:boundarybundle} that $E$ is a holomorphic vector bundle for which we can use $\{\sigma_i(\lambda,x)\}^k_1$ and $\{\eta_i(\lambda)\}^k_1$ to construct local trivializations over the upper and lower hemispheres of $M$, as defined in \S \ref{section:twodim}.  In particular define $\{\eta_i(\lambda,x)\}^k_1$, for $(\lambda,x)\in K \times (0,1]$, as the solutions to equation (\ref{eq:boundaryprob}) which converge to $\{\eta_i(\lambda)\}$ at $x=1$, and suppose we extend the $\{\sigma_i(\lambda,1)\}^k_1$ holomorphically into an open set in $K^\circ$ homotopy equivalent to $S^1$.  
\begin{mydef}
Define the trivializations of the boundary bundle $E$ over open sets in $M$, by
\begin{equation*} \begin{matrix}
\phi_- : & H_- \times \mathbb{C}^k &\hookrightarrow & H_- \times \mathbb{C}^n \\ \\
 & (\lambda,x , ze_i) & \mapsto & \big(\lambda,x, z \sigma_i(\lambda,x)\big) \\ \\
\phi_+ : & H_+ \times \mathbb{C}^k &\hookrightarrow & H_+ \times \mathbb{C}^n \\ \\
 & (\lambda,x , ze_i) & \mapsto & \big(\lambda,x, z \eta_i(\lambda,x)\big) \\
 \end{matrix} \end{equation*}
whereby the transition map at $\{1\}\times K$ is defined by the matrix $\hat{\phi}(\lambda,1,-) := \phi_+^{-1} \circ \phi_-(\lambda,1,-)$.  

\end{mydef}
\begin{lemma}
The winding of the determinant of the transition function,
\begin{equation*}
\det\circ\hat{\phi}_{(1,-)}(\lambda) : K  \rightarrow GL(\mathbb{C}),
\end{equation*}
equals the total multiplicity of eigenvalues of $\mathcal{L}$ contained within $K^\circ$.
\end{lemma}
\begin{proof}
This is proven by Austin \& Bridges \cite{BRI03}.\end{proof}
\subsection{Geometric phase for paths in $E$}
With the bundle view of boundary value problems with $\lambda-$dependent boundary
conditions, we are in the position to utilize the method geometric phase to relate the total multiplicity of eigenvalues contained within $K^\circ$ to the relative phase of paths in the
bundle.  Utilizing the determinant bundle, as in \S\ref{section:extension}, we will recover the Chern number of the boundary bundle $E$ through the relative phase.
The wedge product of the solutions $\{\sigma_i(\lambda,x): i=1,...,k\}$
will form a solution to the associated system on $\Lambda^k\left(\mathbb{C}^n\right)$ $A^{(k)}$ for which we can compute the phase.

\begin{mydef}\label{mydef:boundarybundle} Let $\mathcal{U}(\lambda,x) : = \sigma_1(\lambda,x) \wedge ... \wedge \sigma_k(\lambda,x)$, and denote $\hat{\mathcal{U}}(\lambda,x)$ to be the spherical projection of $\mathcal{U}(\lambda,x)$.  Similarly let $\eta(\lambda,x):= \eta_1(\lambda,x) \wedge ... \wedge \eta_k(\lambda,x)$, and $\hat{\eta}(\lambda,x)$ be the normalization of $\eta$ in the exterior algebra.  Then the line bundle over the parameter sphere with fibers defined by the span of $\mathcal{U}(\lambda,x)$ for $ 0\leq x\leq 1$ and the span of $\eta(\lambda,x)$ for $x=1$ is defined to be the \textbf{determinant bundle of the boundary bundle}.
\end{mydef}
Note that $\hat{\mathcal{U}}$ may not be holomorphic, but as in previous sections it inherits infinite differentiability in the parameter $s$, where $\lambda(s) :[0,1] \hookrightarrow K$.  We may thus calculate the geometric phase of the vector $\hat{\mathcal{U}}\big(\lambda(s),x\big)$ on the Hopf bundle $S^{\left(2{n \choose k} -1\right)}$. By the above definition, we may define trivializations of the determinant bundle similarly to the previous sections through $\hat{\mathcal{U}}(\lambda,x)$ over $H_-$ and $\hat{\eta}(\lambda,x)$ over $H_+$.  The Chern number of this vector bundle is equal to the winding of the transition function, given exactly by the winding of $\det\circ\phi_{(1,-)}(\lambda)$.

\begin{mydef}
The relative phase of $\mathcal{U}(\lambda,x)$, as in Definition \ref{mydef:boundarybundle}, is defined to be the quantity
\begin{align}
GP\Big(\mathcal{U}(K,x)\Big) - GP\Big(\eta(K,1)\Big)
\end{align}
\end{mydef}

\begin{theorem}[\textbf{The method of geometric phase---finite domains}]\label{theorem:boundary}
Let $\mathcal{U}(\lambda,x), \eta(\lambda,x)$ be defined as in Definition \ref{mydef:boundarybundle}; the relative phase of $\mathcal{U}(\lambda,1)$,
\begin{align}
GP\Big(\mathcal{U}(K,1)\Big) - GP\Big(\eta(K,1)\Big),
\end{align}
 is equal to the total multiplicity of the eigenvalues enclosed by the contour $K$ if $\mathcal{U}(\lambda,0)$ and $\eta(\lambda,1)$ are holomorphic and non-zero over $K^\circ$.
\end{theorem}
\begin{proof}
The calculations of the winding of the transition function and the geometric phase are analogous to the calculations performed in \S \ref{section:extension}; there is no difference in calculating the geometric phase and transition function with respect to these trivializations, and the proofs of the lemmas of \S \ref{section:twodim} and \S \ref{section:extension} will also work for the boundary bundle setting. Therefore, the relative phase of $\mathcal{U}(\lambda,x)$ at $x=1$ agrees with the total multiplicity of the eigenvalues contained in $K^\circ$ if the paths $\mathcal{U}(\lambda,0)$ and $\eta(\lambda,1)$ enclose no zeros or poles.\end{proof}
\begin{remark}
Theorem \ref{theorem:boundary} formalizes and validates the method of geometric phase for boundary value problems of the form described in Austin \& Bridges \cite{BRI03}.
\end{remark}

\section{A numerical example}
\label{section:numerics}
In this section we present an example exploring Way's numerical method for computing the geometric phase on the Hopf bundle.  This example illustrates some of the
properties of the phase and its variation along paths, and it is used for exploring
future directions for research regarding the geometric phase and the underlying structure of the travelling wave.  We demonstrate a clear dependence on the length of the integration in the $\xi$ direction, where the \textbf{relative phase} changes continuously from zero
to the value of the multiplicity of the eigenvalue.  The geometric phase of a differentiable path in the unstable manifold is not generically zero, as demonstrated in the examples.  However, for symmetric systems, the relative phase will always transition from zero to the eigenvalue count, by the construction of the relative phase.

Returning to the bi-stable example, equation (\ref{eq:bistable}) with non-linearity $f(u)=u(u+1)(u-1)$, consider the case when $c=0$.  Then $\xi=x$ and $u(\xi)=\sqrt{2}sech(\xi)$ is a time independent solution to the equation $u_t=u_{\xi\xi}-u+u^3$, with $-\infty<\xi<+\infty$.  Consider the linearization $\mathcal{L}$ about the basic state.
 Trivially, $0$ is an eigenvalue of multiplicity one for the linear operator $\mathcal{L}$.  The dynamical systems formulation of $\mathcal{L}$ is
\begin{equation}\begin{matrix}
Y'=A(\lambda , \xi)Y  & & \xi \in (-\infty, \infty)\,,\ \lambda\in\Omega\subset\C\\[4mm]
A = \begin{pmatrix}
0 & 1 \\
\lambda +1 - 6sech^2(\xi) & 0 
\end{pmatrix}
& &
A_{\infty}(\lambda) = \begin{pmatrix}
0 & 1 \\
\lambda +1 & 0
\end{pmatrix}
\end{matrix} \end{equation}
which is equivalent to the operator $\mathcal{L}_\xi(p) = p_{\xi\xi} + f'(u(\xi))p$.  The eigenvalues/vectors for the asymptotic system are of the form
\begin{align}\label{eq:reacdiffunstab}
+\sqrt{\lambda+1}, \hspace{5mm} & \begin{pmatrix} 1 \\  \\  \sqrt{\lambda+1} \end{pmatrix} \\ 
-\sqrt{\lambda+1}, \hspace{5mm} & \begin{pmatrix} 1 \\  \\ -\sqrt{\lambda+1} \end{pmatrix}
\end{align}

Take the contour $K$ to be a circle of radius $0.1$ about the origin in the complex
plane.  In the figures below we plot the building of the geometric phase versus the integration interval in $\xi$.  By discretizing the contour $K$ into $10,000$ even steps, we compute the geometric phase of the forward integrated loop of eigenvectors in equation (\ref{eq:reacdiffunstab}) from $\xi=-11$ to $\xi =11$ with two different scalings.
For each point $\lambda\in K$, the initial condition is integrated forward in $\xi$ and the equation (\ref{eq:computephase}) is computed, with the derivative approximated with the difference
\begin{align*}
\partial_s\mid_{s = s_0} u(s,\xi) \approx& \frac{u(s_0+\delta s,\xi) - u(s_0-\delta s, \xi)}{2\delta s}.
\end{align*}
\noindent \begin{minipage}[t]{.5\linewidth}
\begin{figure}[H]
\includegraphics[width = \linewidth]{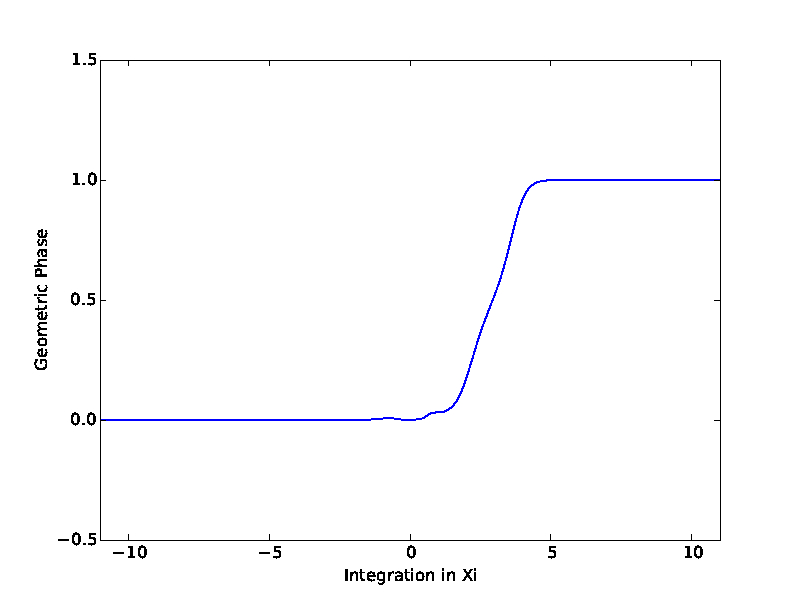}
\caption{Non-degenerate initial conditions.}
\end{figure}
\end{minipage}
\noindent \begin{minipage}[t]{.5\linewidth}
\begin{figure}[H]
\includegraphics[width = \linewidth]{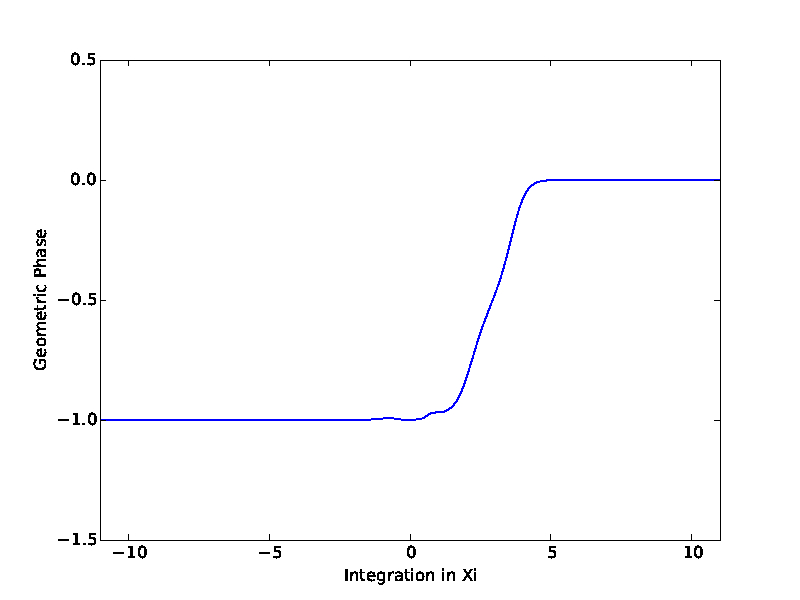}
\caption{Scaled with simple pole at zero.}
\end{figure}
\end{minipage}
 \vspace{\belowdisplayskip}

\noindent The geometric phase experiences a transition near $\xi=0$ in these two examples. Hence, the phase calculation need not be performed for $\xi$ ``close'' to $+\infty$, but simply past a threshold where the change of phase occurs.  In the left figure we take the eigenvectors in equation (\ref{eq:reacdiffunstab}) exactly as our initial conditions, but on the right we instead scale our initial condition by the factor $\frac{1}{\lambda}$, so there is a pole enclosed at $0$.  In the degenerate case on the right, the geometric phase of the initial condition is $-1$, and thus we recover the phase profile translated by the index of the degeneracy.

Although in the above \textbf{non-degenerate} example, the initial geometric phase is zero, it need not be so in general.  The contour $K$ defined as the circle with center at $0.1$ and radius $1$ nears $\lambda=-1$, where $A_{\infty}(\lambda)$ is singular; we discretize this contour into $20,000$ even steps and compute the geometric phase of non-degenerate initial conditions evolved as in the previous example.  For this contour, the geometric phase of the non-degenerate initial conditions in equation (\ref{eq:reacdiffunstab}) has a different profile, beginning with phase greater zero and terminating with phase greater than the eigenvalue count.  
   
\noindent \begin{minipage}[t]{.5\linewidth}
\begin{figure}[H]
\includegraphics[width = \linewidth]{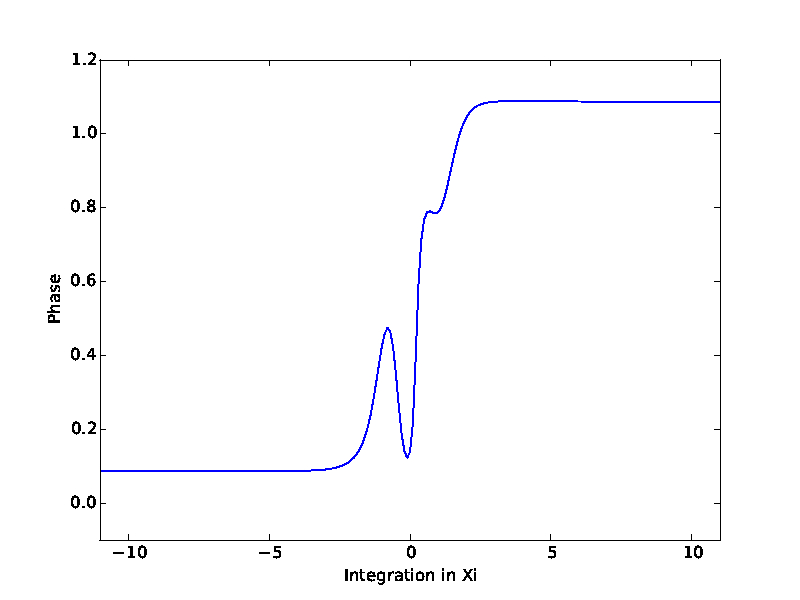}
\caption{The geometric phase profile of the evolved solution.}
\end{figure}
\end{minipage}
\noindent \begin{minipage}[t]{.5\linewidth}
\begin{figure}[H]
\includegraphics[width = \linewidth]{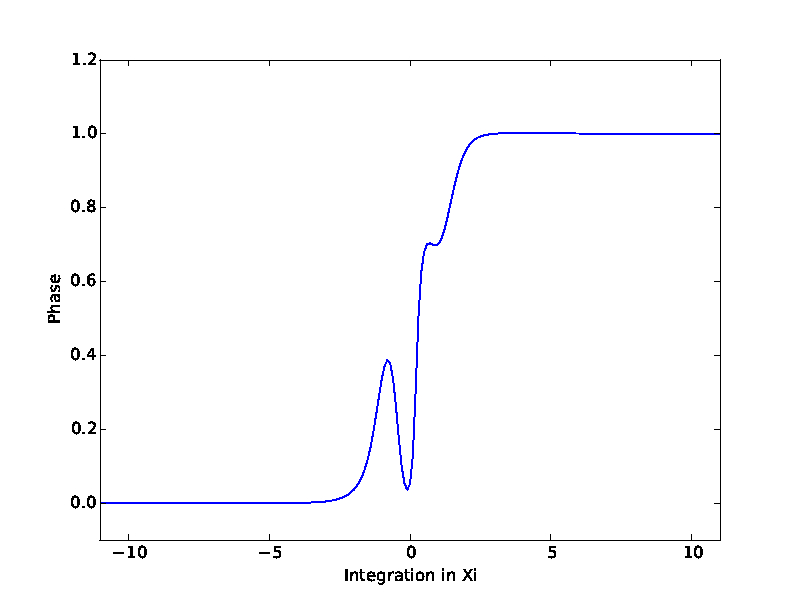}
\caption{The relative phase profile of the evolved solution.}
\end{figure}
\end{minipage}
 \vspace{\belowdisplayskip}

This specific example demonstrates the necessity of the \textbf{relative phase} formulation; in this system with symmetric asymptotic conditions, the relative phase may be formulated as
\begin{align}
GP\Big(Z(K,\tau)\Big) - GP\Big(Z(K,-1)\Big)
\end{align} 
because the reference paths may be chosen $X^-(\lambda) = X^+(\lambda)$.  We plot the relative phase as the terminal geometric phase minus the initial geometric phase; here the relative phase transitions between zero and the eigenvalue count as expected.  This second example also demonstrates the non-uniform nature of the phase transition---additional numerical experiments and discussion on the phase transition in the linearization about waves, with application to
the Hocking-Stewartson pulse solution of
the complex Ginzburg-Landau equation, are given by Grudzien \cite{grudzien-hocking-stewartson}.

\subsection{Discussion}\label{section:discussion}
The numerical examples exhibit a clear dependence on the length of the forward integration in the $\xi$ direction and it suggests firstly that it is not in general necessary to forward integrate the $\lambda$ dependent loop of eigenvectors to a value ``close to $+\infty$'', but rather, past some critical point at which there is a transition between the relative phase equals zero and the relative phase equals the total multiplicity of the eigenvalues enclosed by the spectral path.  Given that the proof of the method equates the relative phase to the Chern number of the bundle over the parameter sphere, it seems intuitive that this should be the case.  Indeed, the Chern number describes a gluing condition for the trivializations of the hemispheres, and the relative phase seems to feel the transition between these trivializations at some intermediate point, rather than at ``$+\infty$''.  Understanding the nature of this transition is of critical importance to the computational method, and the relationship of the phase transition to the underlying wave is currently unclear.

\section{Concluding remarks}
In this paper we have developed the method of geometric phase, inspired by the work of Way \cite{WAY2009}, establishing a relationship between the geometric phase and the Evans function for asymptotically autonomous and boundary value problems.  By reformulating Way's original method into the computation of the relative phase, and building on the relationship between the Chern number and the asymptotic relative phase, we validated the method of geometric phase for systems of ODEs of arbitrary dimension using the determinant bundle, as described by Alexander, Gardner \& Jones \cite{AGJ1990}.  Building on the work of Gardner \& Jones \cite{GJ91} and Austin \& Bridges \cite{BRI03}, we adapted the method of geometric phase to fit $\lambda$ dependent boundary value problems.  In addition, we have presented numerical examples that open up new questions for the study of the method of geometric phase.

\section{Acknowledgements}
Many thanks go to Rupert Way for opening this line of research with \textit{Dynamics in the Hopf bundle, the geometric phase and implications for dynamical systems} \cite{WAY2009}, and for
making available his \textsc{Matlab} codes.  This work benefited from the support of NSF SAVI award DMS-0940363, MURI award A100752 and GCIS award DMS-1312906.

\bibliography{biblio}

\begin{thebibliography}{10}
\expandafter\ifx\csname url\endcsname\relax
  \def\url#1{\texttt{#1}}\fi
\expandafter\ifx\csname urlprefix\endcsname\relax\def\urlprefix{URL }\fi
\expandafter\ifx\csname href\endcsname\relax
  \def\href#1#2{#2} \def\path#1{#1}\fi

\bibitem{WAY2009}
R.~Way, Dynamics in the {H}opf bundle, the geometric phase and implications for
  dynamical systems, Ph.D. thesis, University of Surrey (2009).

\bibitem{berry-paper}
M.~V. Berry, Quantal phase factors accompanying adiabatic changes, Proceedings
  of the Royal Society of London A 392 (1984) 45--57.

\bibitem{2012geometric}
D.~Chruscinski, A.~Jamiolkowski, Geometric Phases in Classical and Quantum
  Mechanics, Progress in Mathematical Physics, Birkh{\"a}user Boston, 2012.

\bibitem{Bates89}
P.~W. Bates, C.~K. Jones, Invariant manifolds for semilinear partial
  differential equations, in: U.~Kirchgraber, H.~Walther (Eds.), Dynamics
  Reported, Vol.~2, Vieweg+Teubner Verlag, 1989, pp. 1--38.

\bibitem{AGJ1990}
J.~Alexander, R.~Gardner, C.~K. R.~T. Jones, A topological invariant arising in
  the stability analysis of traveling waves, Journal fur die Reine und
  Angewandte Mathematik 410 (1990) 167--212.

\bibitem{GJ91}
R.~Gardner, C.~K. R.~T. Jones, A stability index for steady state solutions of
  boundary value problems for parabolic systems, Journal of Differential
  Equations 91 (1991) 181--203.

\bibitem{BRI03}
F.~R. Austin, T.~J. Bridges, A bundle view of boundary-value problems:
  generalizing the {G}ardner\mbox{-}{J}ones bundle, Journal of Differential
  Equations 189 (2003) 412--439.

\bibitem{evans1972}
J.~Evans, Nerve axon equations {II}, Indiana University Mathematics Journal 22
  (1972) 75--90.

\bibitem{evans72}
J.~Evans, Nerve axon equations {I}, Indiana University Mathematics Journal 21
  (1972) 877--885.

\bibitem{evan72}
J.~Evans, Nerve axon equations {III}, Indiana University Mathematics Journal 22
  (1972) 577--593.

\bibitem{evans75}
J.~Evans, Nerve axon equations {IV}, Indiana University Mathematics Journal 24
  (1975) 1169--1190.

\bibitem{kapitula2013}
T.~Kapitula, K.~Promislow, Spectral and Dynamical Stability of Nonlinear Waves,
  Applied Mathematical Sciences, Springer, 2013.

\bibitem{morita2001}
S.~Morita, Geometry of Differential Forms, Vol. 201, American Mathematical
  Society, 2001.

\bibitem{kobayashi1996}
S.~Kobayashi, K.~Nomizu, Foundations of Differential Geometry, no. v. 1 in A
  Wiley Publication in Applied Statistics, Wiley, 1996.

\bibitem{hump06}
J.~Humpherys, B.~Sandstede, K.~Zumbrun, Efficient computation of analytic bases
  in evans function analysis of large systems, Numerische Mathematik 103~(4)
  (2006) 631--642.

\bibitem{allenbridges2002}
L.~Allen, T.~J. Bridges, Numerical exterior algebra and the compound matrix
  method, Numerische Mathematik 92~(2) (2002) 197--232.

\bibitem{grudzien-hocking-stewartson}
C.~Grudzien, The instability of the {H}ocking-{S}tewartson pulse and its
  geometric phase in the {H}opf bundle, arXiv:1504.07896.

\end{thebibliography}
\end{document}